\documentclass[11pt]{article}
\usepackage[english]{babel}					

\usepackage[utf8]{inputenc}					
\usepackage{fullpage}						
\usepackage[T1]{fontenc}						
\usepackage{amsmath}							
\usepackage{amsthm}							
\usepackage{amsfonts}						
\usepackage{amssymb}							
\usepackage{tikz}							
\usepackage{graphicx}						
\usepackage{array}							

\usepackage{hyperref}
\usepackage{subfig}						

\theoremstyle{plain}							
\newtheorem{thm}{Theorem}[section]			
\newtheorem{lem}[thm]{Lemma}					

\newtheorem{propo}[thm]{Proposition}
\newtheorem{Def}[thm]{Definition}

\theoremstyle{definition}					
\newtheorem{rmq}[thm]{Remark}
\newtheorem{notation}[thm]{Notation}

\newcommand{\od}{\partial{}}					
\newcommand{\ud}{\mathrm{d}}					
\newcommand{\RR}{\mathbb{R}^{2}}				
\newcommand{\ee}{\varepsilon}				

\DeclareMathOperator{\indc}{ind}
\newcommand{\CC}{\mathcal{C}}
\newcommand{\WW}{\mathcal{W}}
\newcommand{\adm}{\mathcal{A}}
\newcommand{\OO}{\Omega}
\newcommand{\cer}{\mathbb{S}^{1}}

\begin{document}

\title{Existence, regularity and structure of confined elasticae}


\author{François Dayrens\footnote{Institut Camille Jordan,
Universit\'e Lyon 1,
43, Boulevard du 11 novembre 1918,
69622 Villeurbanne Cedex,
France, e-mail: dayrens@math.univ-lyon1.fr} 
\and Simon Masnou\footnote{Institut Camille Jordan,
Universit\'e Lyon 1,
43, Boulevard du 11 novembre 1918,
69622 Villeurbanne Cedex,
France, e-mail: masnou@math.univ-lyon1.fr} 
\and Matteo Novaga\footnote{Dipartimento di Matematica, Universit\`a di Pisa, Largo Bruno Pontecorvo 5, 56127 Pisa, Italy, e-mail: novaga@dm.unipi.it}}

\maketitle

\begin{abstract}
We consider the problem of minimizing the bending or elastic energy among Jordan curves confined in a given open set $\Omega$. We prove  existence, regularity and some structural properties of minimizers. In particular, when $\Omega$ is convex we show that a minimizer 
is necessarily a convex curve. We also provide an example of a minimizer with self-intersections.
\end{abstract}


\section{Introduction}
The study of curvature-based energies for curves started with Bernoulli and Euler's works on elastic rods and thin beams, see~\cite{Sachkov2008} for a historical overview.

Considering a regular curve $\gamma$ with curvature $\kappa$, its bending energy -- also called Bernoulli-Euler's elastic energy -- is given by
\[\int_\gamma \kappa^2\ud s .\]
Such energy is not only fundamental in the context of mechanics, it is also important in image processing and computer vision, see for instance the applications to amodal completion, image inpainting, image denoising, or shape smoothing in~\cite{MumfordNitzberg,Mumford,Ambrosio2003,EsedogluShen,Bretin2011,Cao2011,Chan2002,Citti2006,Masnou2006,Schoenemann-et-al-tip11,SchoenemannKMC12,Bredies2015,ulen15}.

Curves which are critical points or minimizers of the bending energy are called \emph{elasticae}, and the study of their properties have motivated many contributions, see a few references below. Free minimization is not well-posed, since the energy of a circle with radius $R$ vanishes as $R\to+\infty$, and the only curves with zero bending energy are straight lines. Therefore, 
additional constraints are needed in order to have non-trivial minimizers. 
Frequently the length of the curve is prescribed (see for instance \cite{Avvakumov2013,Sachkov2012}) but other constraints can also be used, as in the non-exhaustive following list:
\begin{itemize}
\item prescribed enclosed area for closed curves in the plane~\cite{Arreaga2002,Bucur2014,Ferone2014};
\item open curves in the plane with clamped ends~\cite{Coope1992,Forsythe1973,Horn1983};
\item confined closed curves, in a phase-field approximation context~\cite{Dondl2011,2015arXiv150701856D};
\item knotted curves in $\mathbb{R}^3$~\cite{Langer1984a};
\item curves lying in a Riemannian manifold~\cite{Koiso1992,Langer1984};
\item obstacle problem for graphs with adhesive terms \cite{Miura2015}.
\end{itemize}

A classical variational issue related to the bending energy is the study of the $L^2$-gradient flow, i.e. the time-dependent evolution which decreases the most the bending energy, see for instance \cite{Linner89,Langer1985} for the general flow, \cite{Okabe2007} for the area-preserving case, and~\cite{Olischlaeger2009} for numerical approximations. 

Another classical issue is the study of the relaxation of the bending energy, i.e. its lower semicontinuous envelop for a suitable notion of convergence for curves or enclosed sets. In particular, the series of papers~\cite{Bellettini1993,Bellettini2004, Bellettini2007} are devoted to the classification of sets in $\RR$ with finite relaxed energy, i.e. sets which are $L^1$-limits of sequences of smooth sets with uniformly bounded energy. Such limits are ubiquitous in the aforementioned references in computer vision and image processing. In particular, it is proved that the (reduced) boundaries of such sets can be covered by a countable collection of $W^{2,2}$ curves with finite elastic  energy and without crossing, i.e. with possible but only tangential autocontacts. 

The problem we tackle in this paper is directly related to such curves: we address the existence, the regularity, and the structure of minimizers of the elastic energy among $W^{2,2}$ curves confined in a prescribed open set $\OO\subset\RR$ and with possible tangential self-intersections. It is worth noticing that such curves are not necessarily limits of Jordan curves, because tangential crossings are allowed. We shall see however that, up to a reparameterization, minimal curves are indeed limits of Jordan curves, so that our problem is indeed equivalent to minimizing the elastic energy among Jordan curves contained in $\OO$.

\smallskip

The paper is organized as follows:  in Section \ref{Definition} we introduce the problem and give basic notation and definitions. Section \ref{existence} deals with the existence of minimal curves using the direct method of calculus of variations. In Section \ref{structure} we study carefully the structure of tangential self-intersections for minimal curves, and we prove that any minimal curve can be reparameterized so that its index is either $0$ or $1$. Section \ref{regularite} is devoted to the regularity of minimal curves, either near contacts points (autocontacts or contacts with the confinement) or on free parts: we prove that a minimal curve $\gamma$ is $W^{3,\infty}$, and that $\gamma^{[3]}$ and $\kappa'$ are functions with bounded variation. Furthermore, we show that $\gamma$ is locally $C^\infty$ away from particular contact points. In Section \ref{convexite} we prove that if the confinement $\OO$ is convex then every minimal curve encloses a convex set. 
Eventually, in Section \ref{example} we provide an example of a non-simply connected confinement in which every minimal curve is not a Jordan curve.

\subsection*{Acknowledgments}
The authors warmly thank Dorin Bucur, Alberto Farina, and Petru Mironescu
for fruitful discussions. FD and SM acknowledge the support of the French Agence
Nationale de la Recherche (ANR) under grant ANR-12-BS01-0014-01 (project
GEOMETRYA). MN was partially supported by the Italian CNR-GNAMPA and by the University of Pisa via the grant PRA-2015-0017.


\section{Preliminary definitions and formulation of the problem} \label{Definition}

We recall a few notions about planar regular curves. Let $\OO$ be a bounded and Lipschitz open set of $\RR$.

\begin{Def}
We denote by $H^2$ the set of all closed curves $\gamma$ whose distributional derivatives satisfy \[ \gamma, \gamma', \gamma'' \in L^2(\cer,\RR) . \]
\end{Def}

\noindent We refer to \cite{Aubin1998} and \cite{Hebey2000} for a general description of Sobolev spaces defined on manifolds. It is enough to recall that the Sobolev space $H^2$ endowed with the norm
\[ ||\gamma||_{H^2} = \left( \int_{\cer} \big( |\gamma(t)|^2 + |\gamma'(t)|^2 + |\gamma''(t)|^2 \big) \ud t \right)^{\frac{1}{2}} \]
(where $|.|$ stands for the Euclidean norm on $\RR$) is a Hilbert space continuously and compactly embedded in $C^{1,\alpha}$ for all $\alpha < 1/2$.

\begin{Def}[Regular curve]
We say that $\gamma : \cer \to \RR$ is a regular curve if $\gamma \in H^2$ and
\[ \forall t\in \cer, \qquad |\gamma'(t)|>0 . \]
\end{Def}

\noindent To every regular curve, we associate its Bernoulli-Euler energy:

\begin{Def}[Bending energy]
Let $\gamma: [0,L] \to \RR$ be a regular curve with arc-length parameterization (thus $L$ denotes the length of $\gamma$). We define its Bernoulli-Euler bending energy by
\begin{equation} \label{expression1}
\WW(\gamma) = \int_0^L |\gamma''(s)|^2 \ud s.
\end{equation}
\end{Def}

\noindent We focus on curves confined in $\OO$.

\begin{Def}[Confinement]
We say that $\gamma$ is confined in $\OO$ if $(\gamma)\subset \overline{\OO}$, where $(\gamma)$ denotes the image of the curve $\gamma$.
\end{Def}

\begin{rmq}
Assume that $\gamma$ is a regular curve confined in $\OO$, and let $x \in \od \OO \cap (\gamma)$ such that $\od \OO$ admits a tangent at $x$. Then, by the confinement condition, $\gamma$ is tangent to $\od \OO$ at $x$.
\end{rmq}

\noindent Lastly, we recall the definition of the index of a closed curve.

\begin{Def}
Let $\gamma : \cer \to \RR=\mathbb{C}$ be an oriented, continuous and piecewise $C^1$ closed curve. For all $z\in\RR\smallsetminus (\gamma)$ we define the index of $\gamma$ around $z$ by
\[ \indc_\gamma (z) = \frac{1}{2i\pi} \int_{\gamma} \frac{\ud w}{w-z}
= \frac{1}{2i\pi} \int_{\cer} \frac{\gamma'(s)}{\gamma(s)-z}\ud s .\]
\end{Def}

The set of Jordan curves is not suitable for minimizing the elastic energy: we will see in Section \ref{example} an explicit example of a confinement set $\OO$ for which minimizers have self-contacts. As in~\cite{Ferone2014,Bucur2014}, we could rather use all $C^1$-limits of Jordan curves. However, in view of the system of curves which appear canonically in~\cite{Bellettini1993,Bellettini2004,Bellettini2007}, we will work with the following larger class of admissible curves:
\begin{Def}
Let $\adm$ denote the set of regular, closed curves confined in $\OO$ with possible tangential self-intersections. Any such curve is called {\it admissible}.
\end{Def}
Clearly, $\adm$ contains the closure of the set of Jordan curves with respect to $C^1$ convergence (see Figure \ref{ensemble}). The inclusion is strict for tangential crossings are allowed. For instance the middle curve in Figure  \ref{ensemble} can be parameterized as a pinched eight.

\begin{figure}[!ht]
\begin{center}

\begin{tikzpicture}
\draw [thick]
	plot [domain=0:6.4,samples=100] ({5+2*cos(\x r)} , {sin(2*\x r)}) ;
\draw [thick]
	plot [domain=0:180,samples=100]
		({+2*sin(\x)}, {-4*sin(2*\x)*(\x/180)*(1-\x/180)})
	plot [domain=0:180,samples=100]
		({-2*sin(\x)}, {-4*sin(2*\x)*(\x/180)*(1-\x/180)}) ;
\draw [thick]
	plot [domain=0:180,samples=100]
		({-5+6*(\x/180)*(1-\x/180)},{+sin(2*\x)*(\x/180)*(1-\x/180)})
	plot [domain=0:180,samples=100]
		({-5-2*sin(3*\x)},{-4*sin(2*\x)*(\x/180)*(1-\x/180)}) ;
\draw (-5,-1.5) node{admissible curve} (0,-1.5) node{admissible curve} (5,-1.5) node {non admissible curve} ;
\end{tikzpicture}

\end{center}
\caption{Admissible curves}
\label{ensemble}
\end{figure}
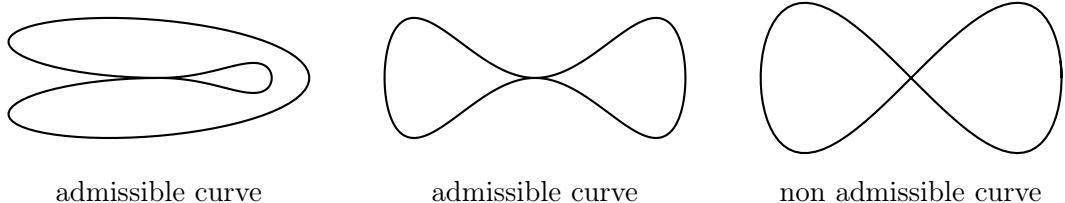

In this paper we study the minimization problem
\begin{equation} \label{problem}
\min_{\gamma \in \adm} \WW(\gamma).
\end{equation}
Clearly, any Jordan curve $\gamma$ can be parameterised so that $\indc_\gamma$ takes values in $\{ 0,1\}$. Furthermore, the index $\indc_{\tilde{\gamma}}$ of any $C^1$--limit ${\tilde{\gamma}}$  of Jordan curves takes values in $\{0,1\}$. Of course, this is no more true in general for curves in $\adm$. However, we will prove in Section \ref{structure} that any curve $\gamma_{min}$ which is minimal for \eqref{problem} (without index restriction) can be reparameterised in order to satisfy $\indc_{\gamma_{min}}(\RR\setminus(\gamma_{min}))\in\{ 0,1\}$.


\section{Existence} \label{existence}

We first prove existence of minimizers for \eqref{problem} using the direct method of calculus of variations. Given a minimizing sequence of curves, we use a constant speed parametrization of the curves defined on $\cer$. This yields a simple expression of the energy (as in \eqref{expression1}) and a constant parametrization space which does not involve the length of each curve. Definition (\ref{expression1}) is modified according to the following obvious proposition:

\begin{propo}
Let $\gamma : \cer \to \RR$ be a regular constant speed parametrized curve with length $L$, then
\begin{equation} \label{expression2}
\WW(\gamma) = \frac{1}{L^3} \int_{\cer} |\gamma''(t)|^2 \ud t .
\end{equation}
\end{propo}

The compactness of a minimizing sequence in $\adm$ follows from a uniform bound on the lengths which we now prove.

\begin{propo} \label{majoration}
For all $M>0$, there exist two constants $C,C'>0$ depending only on $M$ and $\OO$ such that, for all curves $\gamma \in \adm$ satisfying  $\WW(\gamma) \leqslant M$, we have
\[ ||\gamma'||_{L^2} \leqslant C \qquad \text{and} \qquad L \leqslant C'.\]
\end{propo}

\begin{proof} For convenience, we set
\[ \int_{\cer} 1 \ud t = 2\pi \]
and we assume that $\OO \subset B(0,R)$. As $(\gamma) \subset \OO$, for all $t\in \cer \ : \ |\gamma(t)|\leqslant R$ therefore, by the Cauchy-Schwarz inequality,

\begin{equation} \label{gamma}
||\gamma||_{L^2}\leqslant \sqrt{2\pi}R . 
\end{equation}

\noindent Using \eqref{expression2} we have
\begin{equation} \label{gammapp}
||\gamma''||_{L^2}^2 \leqslant ML^3 .
\end{equation}

\noindent In order to control $||\gamma'||_{L^2}$, we use the integration by parts formula :
\[ \int_{\cer} \gamma'(t)\cdot \gamma'(t) \ud t =-\int_{\cer} \gamma''(t)\cdot \gamma(t) \ud t . \]
Thanks to the Cauchy-Schwarz inequality, we have
\begin{equation} \label{gammap}
||\gamma'||_{L^2}^2 \leqslant ||\gamma||_{L^2}||\gamma''||_{L^2} .
\end{equation}
and
\begin{equation} \label{L}
L = \int_{\cer} |\gamma'(t)| \ud t \leqslant \sqrt{2\pi}||\gamma'||_{L^2} .
\end{equation}

\noindent Combining (\ref{gammapp}) and (\ref{L}), one obtains
\begin{equation} \label{gammapp2}
||\gamma''||_{L^2}^2 \leqslant (2\pi)^{\frac{3}{2}}M ||\gamma'||_{L^2}^3 .
\end{equation}

\noindent Using (\ref{gammap}) with (\ref{gamma}) and (\ref{gammapp2}), we deduce that
\[ ||\gamma'||_{L^2}^2 \leqslant (2\pi)^{\frac{5}{4}}R\sqrt{M}||\gamma'||_{L^2}^{3/2}, \]
therefore
\begin{equation} \label{gammap2}
||\gamma'||_{L^2} \leqslant (2\pi)^{\frac{5}{2}}R^2M.
\end{equation}
Eventually, using \eqref{L}
\[ L \leqslant (2\pi)^3R^2M .\]

\end{proof}

Having a compactness property for our minimization problem, we can prove the following theorem.

\begin{thm}
The minimization problem
\[ \min_{\gamma \in \adm } \WW(\gamma) \]
admits a solution.
\end{thm}

\begin{proof} \

\textbf{Step 1:} Competitor for minimality. \\
Let $(\gamma_n)_n$ be a minimizing sequence such that, without loss of generality, $\sup_n \WW(\gamma_n)\leq M$ for some $M>0$.
Proposition \ref{majoration} guarantees a uniform bound on $(\gamma_n)_n$ with respect to the $H^2$ norm. Plugging \eqref{gammap2} into \eqref{gammapp2}, one gets
\begin{equation} \label{gammapp3}
||\gamma_n''||_{L^2}^2 \leqslant (2\pi)^9R^6M^4 .
\end{equation}

\noindent Using (\ref{gamma}), (\ref{gammap2}) and (\ref{gammapp3}) we calculate that
\[ ||\gamma_n||_{H^2} = \sqrt{||\gamma_n||^2_{L^2} + ||\gamma_n'||^2_{L^2} + ||\gamma_n''||^2_{L^2} } \leqslant
\sqrt{ 2\pi R^2+(2\pi)^5 M^2R^4+(2\pi)^9M^4R^6 } .\]

Then, by Sobolev compact embedding, there exists a curve $\gamma : \cer \to \RR$ and a subsequence $(\gamma_{n_k})_k$ such that
\[ \gamma_{n_k} \longrightarrow \gamma \]
weakly in $H^2$, strongly in $C^1(\cer,\RR)$.

Now we have to check that $\gamma \in \adm$ and its energy is minimal. First we prove that the limit curve $\gamma$ is not degenerate, i.e. not reduced to a point.

\medskip

\textbf{Step 2:} Minimal length. \\
We can slightly modify the minimizing sequence to prescribe a lower bound on its length. Let $\alpha>0$ and $x\in \OO$ such that $B(x,\alpha)\subset\OO$. Let us prove that, for all $n$, $L(\gamma_n)\geqslant \alpha$. Indeed, for $n$ such that $L(\gamma_n)<\alpha$, consider $\tilde{\gamma}_n$ the curve given by a translation of $\gamma_n$ satisfying $x \in \text{im }\tilde{\gamma}_n$. By a length argument, $\tilde{\gamma}_n \subset B(x,\alpha) \subset \OO$. Then let $H_n$ be defined for all $y\in \RR$ by
\[ H_n(y) = x+(1+\ee_n)(y-x) . \]
$\ee_n>0$ is chosen to saturate $H_n(\tilde{\gamma}_n) \subset \overline{B}(x,\alpha)$, i.e. $H_n(\tilde{\gamma}_n) \cap \od B(x,\alpha) \neq \emptyset$. We use (\ref{expression2}) to compute
\[ \WW(H_n(\tilde{\gamma}_n)) = \frac{1}{1+\ee_n} \WW(\gamma_n) < \WW(\gamma_n) . \]
As $H_n(\tilde{\gamma}_n)$ links $x$ and a point of the circle $\od B(x,\alpha)$, $L(H_n(\tilde{\gamma}_n)) \geqslant \alpha$, and one can replace $\gamma_n$ by $H_n(\tilde{\gamma}_n)$ in the minimizing sequence. With no loss of generality, we therefore assume that $L(\gamma_n)\geqslant \alpha$.

\medskip

\textbf{Step 3:} Admissibility and minimality. \\
The strong convergence in $C^1$  of the minimizing sequence leads to uniform convergence of $(\gamma_n)$ and $(\gamma_n')$, thus we recover a limit curve $\gamma$ with  constant speed parametrization, positive length, possible tangential self-contacts and confinement in $\OO$. Therefore $\gamma \in \adm$. Moreover, thanks to uniform convergence of $(\gamma_n')$, we have that $L(\gamma_n) \to L(\gamma)>0$. Then using the lower semi-continuity of the $L^2$ norm with respect to  weak convergence, we have
\[ \liminf_{n\to +\infty} \frac{1}{L(\gamma_n)^3} ||\gamma_n''||^2_{L^2}
\geqslant \frac{1}{L(\gamma)^3} ||\gamma''||^2_{L^2} . \]
Therefore, in view of (\ref{expression2}), $\WW$ is lower semi-continuous, and $(\gamma_n)$ being a minimizing sequence,
\[ \WW(\gamma) = \min_{\tilde{\gamma}} \WW(\tilde{\gamma}) . \]

\end{proof}

\subsection{Contact with the confinement}

Using simple geometric transformations, we can prove that a minimal curve has at least two contact points with the boundary of the confinement $\OO$.

\begin{thm}
Let $\gamma \in \adm$ be a minimal curve. There exist two distinct times $t,s \in \cer$ such that $\gamma(s) \neq \gamma(t)$ and $\gamma(t),\gamma(s) \in \od \OO$.
\end{thm}

\begin{proof}
Assume first that $\gamma$ has no contact with $\od \OO$. Then, by compactness, $d(\gamma, \od \OO) >0$. Let $\lambda>0$ and consider $H_\lambda$ a homothety with rapport $1+\lambda$. By continuity, $d(H_\lambda(\gamma), \od \OO) > 0$ for $\lambda$ small enough. In addition, $H_\lambda(\gamma)\in \adm$ for it has the same regularity and self-contact properties as $\gamma$. However,
\[ \WW(H_\lambda(\gamma)) = \frac{1}{1+\lambda} \WW(\gamma) < \WW(\gamma) .\]
Thus $H_\lambda(\gamma)$ has a smaller energy than the minimal curve, and that is impossible.

Secondly assume $\gamma$ makes only one contact with $\od \OO$. Let $\gamma(t) \in \od \OO$ and write locally $\gamma$ and $\od \OO$ as graphs of functions. Then there exists a coordinate systems $(O,\vec{i},\vec{j})$ such that
\begin{itemize}
\item $O=\gamma(t)$,
\item $\vec{i}$ and $\gamma'(t)$ coincide (up to a positive constant),
\item there exists an open subset $\mathcal{U}$ such that $(\gamma) \cap \mathcal{U} = (\gamma_{|[t_1,t_2]})$ and $\od \OO \cap \mathcal{U} = \Gamma$ are written as graphs of functions (with $t\in [t_1,t_2]$),
\item in the local coordinate system, $\Gamma \leqslant \gamma_{|[t_1,t_2]}$.
\end{itemize}
Let $\tau>0$ and consider $T_\tau$ the translation with respect to $\tau\vec{j}$. As before, $d\Big( (\gamma) \smallsetminus (\gamma_{|]t_1,t_2[}) , \od \OO \Big) > 0$ so $d\Big( (T_\tau(\gamma)) \smallsetminus (T_\tau(\gamma_{|]t_1,t_2[})) , \od \OO \Big) > 0$ for $\tau$ small enough.
Moreover, in the local coordinates, $\Gamma \leqslant \gamma_{|[t_1,t_2]} < T_\tau(\gamma_{|]t_1,t_2[})$, thus
$d(T_\tau(\gamma), \od \OO) >0$, 
and $\WW(T_\tau(\gamma)) = \WW(\gamma)$ should be minimal. Thanks to the previous case, that leads to a contradiction.

Eventually, assume there are several parts of $\gamma$ making contact with the same boundary point $z\in \od \OO$ (but no pair $(t,s)$ such that $\gamma(s) \neq \gamma(t)$ and $\gamma(t),\gamma(s) \in \od \OO$). One can repeat the above construction representing only $\od \OO$ as a graph such that all points of $\gamma \cap \mathcal{U}$ are above $\Gamma$ in the local coordinate system. Using a suitable vertical translation, we obtain that, for $\tau$ small enough, $T_\tau (\gamma)$ remains far from $\od \OO$ outside $\mathcal{U}$, and $T_\tau (\gamma)$ remains strictly above $\Gamma$ inside $\mathcal{U}$. Thus
\[ d(T_\tau(\gamma), \od \OO) >0 \]
and the conclusion is the same.

Consequently, $\gamma$ admits at least two contact points with  $\od \OO$.
\end{proof}


\section{Tangential self-intersections of minimal curves} \label{structure}

In this section, we study the self-intersections of minimizers of $\WW$. We keep using the constant speed parametrization and we denote as $\CC$ the set of self-intersections of $\gamma$, i.e.
\[ \CC=\{ t \in \cer \ | \ \exists s \in \cer, \ s\neq t \ : \ \gamma(t) = \gamma(s) \} . \]
As the self-intersections are tangent, for all $t,s$ such that $\gamma(t)=\gamma(s)$ we have that $\gamma'(t)=\pm \gamma'(s)$.

\subsection{Decomposition of minimal curves}

The following theorem is proved using simple geometric arguments. Locally at a self-contact point $\gamma(t)$ of a minimal curve $\gamma$, there are several branches of $\gamma$ departing from $\gamma(t)$ with two possible directions $\gamma'(t)$ or $-\gamma'(t)$. The recurrent idea of the following proof is based on the observation that, in some configurations, one can start from $\gamma(t)$, run along $(\gamma)$ avoiding some pieces, and yet do a cycle to reach $\gamma(s)=\gamma(t)$ such that $\gamma'(s)=\gamma'(t)$ (see Figure \ref{principe f1}).

\begin{figure}[!ht]
\begin{center}

\begin{tikzpicture}
\draw [thick] plot [domain=0:2] ({-5+1.3*(\x-1/2)},{0.4-0.4*sin(3.1415*\x/1.3 r)}) ;
\draw [thick] plot [domain=-1:1] ({-5+\x},{-\x*\x}) ;
\draw [<-] plot [domain=0.3:1] ({-5+1.3*(\x-1/2)},{0.8-0.4*sin(3.1415*\x/1.3 r)}) ;
\draw [thick,<-] plot [domain=1.4:1.5] ({-5+1.3*(\x-1/2)},{0.4-0.4*sin(3.1415*\x/1.3 r)}) ;
\draw [thick,<-] plot [domain=0.1:0.2] ({-5+1.3*(\x-1/2)},{0.4-0.4*sin(3.1415*\x/1.3 r)}) ;

\draw [thick] plot [domain=0:2] ({5+1.3*(\x-1/2)},{0.4-0.4*sin(3.1415*\x/1.3 r)}) ;
\draw [thick] plot [domain=-1:1] (5+\x,{-\x*\x}) ;
\draw plot [domain=0:1] (5+\x*1.3,{\x*\x/(1.3*1.3)-0.3}) ;
\draw [<-] plot [domain=-0.5:0] ({5+\x},{-0.3-\x*\x}) ;
\draw [thick,<-] plot [domain=1.4:1.5] ({5+1.3*(\x-1/2)},{0.4-0.4*sin(3.1415*\x/1.3 r)}) ;
\draw [thick,<-] plot [domain=0.1:0.2] ({5+1.3*(\x-1/2)},{0.4-0.4*sin(3.1415*\x/1.3 r)}) ;
\draw [thick,->] plot [domain=-0.8:-0.7] ({5+\x},{-\x*\x}) ;
\draw [thick,->] plot [domain=0.7:0.8] ({5+\x},{-\x*\x}) ;

\draw [thick] plot [domain=0:2] ({1.3*(\x-1/2)},{0.4-0.4*sin(3.1415*\x/1.3 r)}) ;
\draw [thick] plot [domain=-1:1] (\x,{-\x*\x}) ;
\draw plot [domain=0:1] (\x*1.3,{\x*\x/(1.3*1.3)-0.3}) ;
\draw [<-] plot [domain=-0.5:0] ({\x},{-0.3-\x*\x}) ;
\draw [thick,<-] plot [domain=1.4:1.5] ({1.3*(\x-1/2)},{0.4-0.4*sin(3.1415*\x/1.3 r)}) ;
\draw [thick,<-] plot [domain=0.1:0.2] ({1.3*(\x-1/2)},{0.4-0.4*sin(3.1415*\x/1.3 r)}) ;
\draw [thick,<-] plot [domain=-0.8:-0.7] ({\x},{-\x*\x}) ;
\draw [thick,<-] plot [domain=0.7:0.8] ({\x},{-\x*\x}) ;
\end{tikzpicture}

\end{center}
\caption{Branching possibilities}
\label{principe f1}
\end{figure}
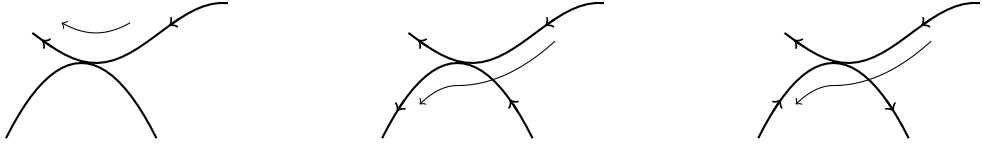

In some configurations $\gamma$ can be decomposed into two parts: an admissible sub-part $\tilde{\gamma}\in\adm$ and the rest $\gamma_u$, which may be a union of possibly open curves. If $\gamma_u$ is non empty and is not composed of segments then $\displaystyle \int_{(\gamma_u)}\kappa^2d{\mathcal H}^1>0$. Using
\[ \WW(\gamma) = \underbrace{\int_{(\gamma_u)}\kappa^2d{\mathcal H}^1}_{>0}
  + \underbrace{\int_{(\tilde\gamma)}\kappa^2d{\mathcal H}^1}_{=\WW(\tilde{\gamma})}, \]
we can see that $\gamma$ cannot be minimal. 

\begin{Def}
Let $\gamma : \cer \to \RR$ be an admissible closed curve.
\begin{itemize}
		\item Given an orientation of $\cer$, we say that $t_1, t_2, t_3$ are ordered times if, up to a circular permutation, $t_1<t_2<t_3$.
		\item Given two ordered times $t<s$, we denote by $\gamma_{|]t,s[}$ the restriction of $\gamma$ starting from $\gamma(t)$, parametrized with the same orientation as the ordered times, and ending at $\gamma(s)$. We denote by $\gamma_{-|]t,s[}$ the restriction of the curve $\gamma$ starting from $\gamma(t)$, parametrized with the inverse orientation, and ending at $\gamma(s)$.
		\item We denote by $\gamma_1 \oplus \gamma_2$ the curve obtained by concatenation of $\gamma_1$ and $\gamma_2$, i.e. if $\gamma_1:[a_1,b_1]\to\RR$ and $\gamma_2:[a_2,b_2]\to\RR$ then $\gamma_1 \oplus \gamma_2:[a_1,b_1+b_2-a_2]\to\RR$ is defined by $\gamma_1 \oplus \gamma_2(t)=\gamma_1(t)$ on $[a_1,b_1]$ and $\gamma_1 \oplus \gamma_2(t)=\gamma_2(t+a_2-b_1)$ on $[b_1,b_1+b_2-a_2]$.
		\item We say that $(t,s)$ is a self-contact couple if $\gamma(t)=\gamma(s)$ and $t< s$.
		\end{itemize}
\end{Def}

\noindent Notice that in the second item, $\gamma_{|]t_3,t_1[}$ has the same orientation as the ordered times and thus its parameter space does not contain $t_2$. Moreover, $\gamma_{|]t,s[}$ and $\gamma_{-|]t,s[}$ represent both distinct branches of $\gamma$ starting from $\gamma(t)$ and ending at $\gamma(s)$.

\begin{thm}[Decomposition] \label{decomposition}
Let $\gamma$ be a minimal curve for \eqref{problem}. Then every self-contact point has multiplicity not greater than $2$ i.e. every self-intersection is associated with its unique self-contact couple. In addition, if $\gamma$ admits at least two self-intersections then there exist four ordered times $t<s<\tau<\sigma$ such that:
\begin{itemize}
		\item $(t,s)$ and $(\tau,\sigma)$ are self-contact couples;
		\item $\gamma_{|]t,s[}$ and $\gamma_{|]\tau,\sigma[}$ are pieces of curve without contact points (i.e. $]t,s[\cap\mathcal{C}$ and $]\tau,\sigma[\cap\mathcal{C}$ are empty) and they possibly have contact with $\od \OO$;
		\item $\gamma_{|]s,\tau[}$ and $\gamma_{|]\sigma,t[}$ are curves without self-contact and with same energy. They possibly have tangent contact to each other and in that case for all $a \in ]s,\tau[$ and for all $b \in ]\sigma,t[$ such that $\gamma(a)=\gamma(b)$, we have
			\[ \WW(\gamma_{|[s,a]}) = \WW(\gamma_{|[b,t]}) \quad \text{and} \quad \WW(\gamma_{|[a,\tau]}) = \WW(\gamma_{|[\sigma,b]})  . \]
\end{itemize}

\end{thm}

\begin{proof}
The proof is made in several small steps using different configurations. Let $\gamma$ be a minimal curve.

\textbf{Step 1:} Uniqueness of self-contact couples. \\
Assume there exist three distinct times $t_1<t_2<t_3$ such that $\gamma(t_1)=\gamma(t_2)=\gamma(t_3)$. Two of them have the same tangent vector, say $t_1$ and $t_2$. Then $\tilde{\gamma} = \gamma_{|[t_1,t_2]}$ is an admissible sub-part of $\gamma$, thus $\WW(\tilde{\gamma}) <  \WW(\gamma)$.
That is impossible because of the minimality of $\gamma$. Thus any self-intersection of $\gamma$ is associated with a unique pair of times $(t,s)$, and with the same argument as above, they have necessarily opposite orientations, i.e. $\gamma'(t) = - \gamma'(s)$. Consequently, a minimal curve does not have self-contact points with multiplicity higher than $2$.

\textbf{Step 2:} $\CC$ is closed. \\
Consider the continuous application
\[ \begin{array}{cccc}
G : & \cer \times\ \cer & \to & \RR\times\RR  \\
 & (s,t) & \mapsto & (\gamma(s)-\gamma(t),\gamma'(s)+\gamma'(t)) . \end{array} \]
We have $\CC = p(G^{-1}(\{0\}))$ with $p:\cer\times\cer \to \cer$ the projection on the first component. Being $G$ continuous and $p$ closed, we deduce that $\CC$ is closed.

\par\textbf{Step 3:} Self-contact couples are non crossing. \\
Given two self-contact couples $(t_1,s_1)$ and $(t_2,s_2)$, we can order them as $t_1<s_1<t_2<s_2$ (up to circular permutation). Indeed, the other possible configuration $t_1<t_2<s_1<s_2$ yields a decomposition into $\tilde{\gamma} = \gamma_{|]t_1,t_2[} \oplus \gamma_{-|]s_2,s_1[}$ (by construction $\tilde\gamma$ is $C^1$) and $(\gamma_u)$ = $(\gamma_{|]s_2,t_1[}) \cup (\gamma_{|]t_2,s_1[})$, which contradicts the minimality of $\gamma$.

\textbf{Step 4:} Self-contact couples are nested. \\
Given three self-contact couples $(t_1,s_1)$, $(t_2,s_2)$ and $(t_3,s_3)$, we can order them as $t_1 < t_2 < t_3 < s_3 < s_2 < s_1$.
Indeed, each pair of self-contact couples is non crossing, thus the only remaining possibilities are either the one above or $t_1 < t_2 < s_2 < t_3 < s_3 < s_1$. Let us show that the second case contradicts minimality. Choose the sub-part of $\gamma$ with least energy among $\gamma_{|]t_1,t_2[}$, $\gamma_{|]s_2,t_3[}$ and $\gamma_{|]s_3,s_1[}$. For convenience, suppose it is $\gamma_{|]t_1,t_2[}$. Consider the $C^1$ curve $\tilde{\gamma} = \gamma_{|]s_1,s_2[} \oplus \gamma_{-|]t_2,t_1[}$. By construction, $\tilde{\gamma}$ has less energy than $\gamma$, see Figure~\ref{trigoute}.

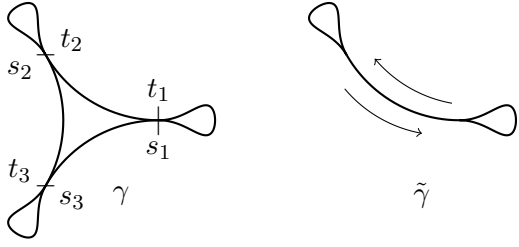
\begin{figure}[!ht]
\begin{center}

\begin{tikzpicture}
\draw[thick]
	plot [domain=0:180,samples=100]
		({1+3*(\x/180)*(1-\x/180)},{sin(2*\x)*(\x/180)*(1-\x/180)})
	node{$|$}
	+(-0.5,-1) node{$\gamma$}
	+(0,0.4) node{$t_1$}
	+(0,-0.4) node{$s_1$}
	plot[domain=0:60] ({1-sqrt(3)*sin(\x)},{sqrt(3)*(1-cos(\x))}) ;

\draw[rotate=120, thick]
	plot [domain=0:180,samples=100]
		({1+3*(\x/180)*(1-\x/180)},{sin(2*\x)*(\x/180)*(1-\x/180)})
	node{$-$}
	+(0,0.4) node{$s_2$}
	+(0,-0.4) node{$t_2$}
	plot[domain=0:60] ({1-sqrt(3)*sin(\x)},{sqrt(3)*(1-cos(\x))}) ;

\draw[rotate=240, thick]
	plot [domain=0:180,samples=100]
		({1+3*(\x/180)*(1-\x/180)},{sin(2*\x)*(\x/180)*(1-\x/180)})
	node{$-$}
	+(0,0.4) node{$s_3$}
	+(0,-0.4) node{$t_3$}
	plot[domain=0:60] ({1-sqrt(3)*sin(\x)},{sqrt(3)*(1-cos(\x))}) ;

\draw[shift={(4,0)}, thick]
	plot [domain=0:180,samples=100]
		({1+3*(\x/180)*(1-\x/180)},{sin(2*\x)*(\x/180)*(1-\x/180)})
	+(-0.5,-1) node{$\tilde{\gamma}$}
	plot[domain=0:60] ({1-sqrt(3)*sin(\x)},{sqrt(3)*(1-cos(\x))}) ;

\draw[shift={(4,0)}, rotate=120, thick]
	plot [domain=0:180,samples=100]
		({1+3*(\x/180)*(1-\x/180)},{sin(2*\x)*(\x/180)*(1-\x/180)}) ;

\draw[shift={(4.2,0.2)}, ->] plot[domain=10:50] ({1-sqrt(3)*sin(\x)},{sqrt(3)*(1-cos(\x))}) ;
\draw[shift={(3.8,-0.2)}, <-] plot[domain=10:50] ({1-sqrt(3)*sin(\x)},{sqrt(3)*(1-cos(\x))}) ;
\end{tikzpicture}

\end{center}
\caption{Example of non-nested self-contact couples}
\label{trigoute}
\end{figure}

\textbf{Step 5:} Special self-contact couples. \\
If $\gamma$ has at least two self-intersections, then there exist two unique self-contact couples $(t,s)$ and $(\tau,\sigma)$ such that $t<s<\tau<\sigma$, and $]t,s[\cap \ \CC$ and $]\tau,\sigma[\cap \ \CC$ are empty.
Indeed if $\gamma$ admits two distinct self-contact couples $(t_0,s_0)$ and $(\tau_0,s_0)$ then they are non-crossing. We order them so that $t_0 < s_0 < \tau_0 < \sigma_0$. We prove the claim for $[t_0,s_0]$, being the argument identical for $[\tau_0,\sigma_0]$. If $]t_0,s_0[\cap \ \CC \neq \emptyset$, consider a point of $\CC$ in $]t_0,s_0[$. By Step 3 the corresponding self-contact couple cannot cross $(t_0,s_0)$ therefore it is included in $]t_0,s_0[$. We denote by $(t_i,s_i)_{i\in I}$ the set of all self-contact couples included in $[t_0,s_0]$. Using Step 4, for all $i,j \in I, \ [t_i,s_i] \subset [t_j,s_j]$ or $[t_i,s_i] \subset [t_j,s_j]$. Considering $t=\sup_{i\in I} t_i$ and $s=\inf_{i\in I} s_i$, one obviously has $t\leqslant s$. Moreover, as $\CC$ is closed (Step 2), $t$ and $s$ are self-contact points. Finally, passing to the limit in $\gamma(t_i)=\gamma(s_i)$ and in $\gamma'(t_i)=-\gamma'(s_i)$, the $C^1$-regularity of $\gamma$ yields $t<s$ and thus $(t,s)$ is a self-contact couple included in $]t_0,s_0[$. By construction, $]t,s[ \cap \ \CC = \emptyset $. Similarly, we prove that there exist $\tau$ and $\sigma$ such that $]\tau,\sigma[ \cap \ \CC = \emptyset $.

\textbf{Step 6:} Branches energies \\
The minimality of $\gamma$ ensures that all paths between $\gamma_{|]t,s[}$ and $\gamma_{|]\tau,\sigma[}$ have the same "least" energy. By path we mean a choice of branches between $\gamma_{|]s,\tau[}$ and $\gamma_{|]\sigma,t[}$ at each intersection point allowing a regular link
between $\gamma_{|]t,s[}$ and $\gamma_{|]\sigma,t[}$. The last item of the theorem is just a reformulation of this "equal energy" property.
\end{proof}

\subsection{Index of a minimal curve}
By definition, curves of $\adm$ may have tangential crossings, therefore arbitrary indices. We will prove that a minimal curve $\gamma$ can be reparametrized in order to satisfy
\[ \indc_\gamma\big( \RR\smallsetminus (\gamma) \big) \subset \{0,1\} . \]
The core of the proof is the following technical lemma.

\begin{lem} \label{parametr}
We can choose a constant speed parametrization of $\gamma$ such that for all distinct $\alpha,\beta \in \cer$, $\gamma_{|]\alpha,\beta[}$ and $\gamma_{-|]\alpha,\beta[}$ may have tangential intersections but do not cross (i.e. remain in the same side one with respect to the other).
\end{lem}

\begin{proof}
The lemma is clear if $\gamma$ is a simple curve. If $\gamma$ admits only one self-contact couple $(t,s)$ then we can write locally at $\gamma(t)=\gamma(s)$ (on what we call below a "local neighborhood") both branches as graphs of functions (called $f$ and $g$) over their common tangent line. Using $\max(f,g)$ and $\min(f,g)$, we can reparametrize $\gamma$ (actually, we reparametrize only one half of $\gamma$, either $\gamma_{|]t,s[}$ or $\gamma_{-|]t,s[}$) in order to keep one branch above the other. $f$ and $g$ belong to $H^2$ and as the contact is tangent then $\max(f,g)$ and $\min(f,g)$ are also in $H^2$. Thus the lemma is proved in that case.

If $\gamma$ has several self-intersections then we construct the parametrization by induction using Theorem \ref{decomposition}. We start from one drop $\gamma_{|]t,s[}$ and we consider $\mathcal{T}$ a tubular neighborhood of the (simple non-closed) curve $\gamma_{|]s,\tau[}$. We can define two sides of $\mathcal{T}$ (one "above" $\gamma_{|]s,\tau[}$ and the other "below"). We construct locally the parametrization in the same way as before using $\min$ and $\max$ functions.

A key point of the construction is that, thanks to Step 4 and 5 in the proof of Theorem \ref{decomposition}, the self-intersections have the same order along $\gamma_{|]s,\tau[}$ and $\gamma_{|-]t,\sigma[}$. It is easily seen that there exists a finite number of local neighborhoods along $\gamma_{|]s,\tau[}$, and on each of them the same reparametrization as above can be applied. More precisely, the construction must be coherent from a local neighborhood to the next one, following $\gamma_{|]s,\tau[}$ and $\gamma_{|-]t,\sigma[}$ and never coming back. In other words, at every step $n$ of the induction around a self-contact couple $(t_n,s_n)$ with $t_n \in ]s,\tau[$ and $s_n \in ]\sigma,t[$, $\gamma_{|]s,t_n[}$ and $\gamma_{|-]t,s_n[}$ have been parametrized in previous steps, and the construction changes only the parametrization around $\gamma(t_n)=\gamma(s_n)$ and the parametrization of $\gamma_{|]t_n,\tau[}$ and $\gamma_{|-]s_n,\sigma[}$. Therefore step $n$ of the construction does not change step $k$ for every $k<n$ (see Figure \ref{tubulaire}). Since there are finitely many local neighborhoods, the induction has a finite number of steps. At the end of the construction, each time two branches of $\gamma$ make a contact then they meet and separate in the same side. Thus the lemma is proved.

\begin{figure}[!ht]
\begin{center}

\begin{tikzpicture}
\fill [color=gray!40] 
	plot [domain=-1:1,samples=100]
		({5*\x + 0.5*(-(10*(\x)^4-12*(\x)^2+2)/(sqrt(25+(10*(\x)^4-12*(\x)^2+2)^2)))},
		{2*(\x-1)^2*(\x+1)^2*(\x) + 0.5*(5)/(sqrt(25+(10*(\x)^4-12*(\x)^2+2)^2)))}) -- 
	plot [domain=-1:1,samples=100]
		({5*(-\x) - 0.5*(-(10*(\x)^4-12*(\x)^2+2)/(sqrt(25+(10*(\x)^4-12*(\x)^2+2)^2)))},
		{2*(\x-1)^2*(\x+1)^2*(-\x) - 0.5*(5)/(sqrt(25+(10*(\x)^4-12*(\x)^2+2)^2)))}) --
	cycle ;
\draw [thick]
	plot [domain=-1:1,samples=100] ({5*\x}, {2*(\x-1)^2*(\x+1)^2*(\x)})
	plot [domain=-0.75:0.75,samples=100]
	({5*\x}, {2*(\x-1)^2*(\x+1)^2*(\x) +4*(\x-0.75)^2*(\x+0.75)^2}) ;
\foreach \t in {-1,-0.9,-0.8,-0.7,0.7,0.8,0.9,1}
	{\draw (5*\t,{2*(\t-1)^2*(\t+1)^2*(\t)}) circle (0.4) ;}

\draw [thick]
	plot [domain=0:180,samples=100] ({5+2*sin(\x)}, {-4*sin(2*\x)*(\x/180)*(1-\x/180)})
	plot [domain=0:180,samples=100] ({-5-2*sin(\x)}, {-4*sin(2*\x)*(\x/180)*(1-\x/180)}) ;
\fill [color=gray!40] 
	plot [domain=0:180,samples=100] ({-3+2*cos(\x)}, {-3+sin(\x)}) --
	plot [domain=-1:1,samples=100] ({-3+2*\x}, {-3-0.5*(\x-1)^2*(\x+1)^2}) -- cycle ;
\fill [color=gray!40] 
	plot [domain=0:180,samples=100] ({3+2*cos(\x)}, {-3+sin(\x)}) --
	plot [domain=-1:1,samples=100] ({3+2*\x}, {-3-0.5*(\x-1)^2*(\x+1)^2}) -- cycle ;

\draw [->] (-3.5,-0.7) -- (-3,-1.8) ;
\draw plot [domain=0:360,samples=100] ({-3+2*cos(\x)}, {-3+sin(\x)}) ;
\draw [thick]
	plot [domain=-1.1:1.1,samples=100] ({-3+2*\x}, {-3-0.5*(\x-1)^2*(\x+1)^2})
	plot [domain=-0.5:1,samples=100] ({-3+2*\x}, {-3-0.5*(\x-1)^2*(\x+1)^2 +0.3*(\x+0.5)^2 }) ;
\draw [dotted] plot [domain=-45:45,samples=100] ({-6+2*cos(\x)}, {-3+sin(\x)}) ;
\draw [dotted] plot [domain=135:225,samples=100] ({6+2*cos(\x)}, {-3+sin(\x)}) ;

\draw [->] (3.5,0) -- (3,-1.8) ;
\draw plot [domain=0:360,samples=100] ({3+2*cos(\x)}, {-3+sin(\x)}) ;
\draw [thick]
	plot [domain=-1.1:1.1,samples=100] ({3+2*\x}, {-3-0.5*(\x-1)^2*(\x+1)^2})
	plot [domain=-1:0.5,samples=100] ({3+2*\x}, {-3-0.5*(\x-1)^2*(\x+1)^2 +0.3*(\x-0.5)^2 }) ;

\draw (-3,-2.5) node{step $n$} (3,-2.5) node{step $n+1$} ;
\draw (-6.5,-1.1) node{$\gamma_{|]t,s[}$} (6.5,-1.1) node{$\gamma_{|]\tau,\sigma[}$} ;
\draw (-1,1.3) node{$\gamma_{|-]t,\sigma[}$} (-4,0.5) node{$\mathcal{T}$} (1,0.7) node{$\gamma_{|]s,\tau[}$} ;
\end{tikzpicture}

\end{center}
\caption{Tubular neighborhood and local graphs}
\label{tubulaire}
\end{figure}
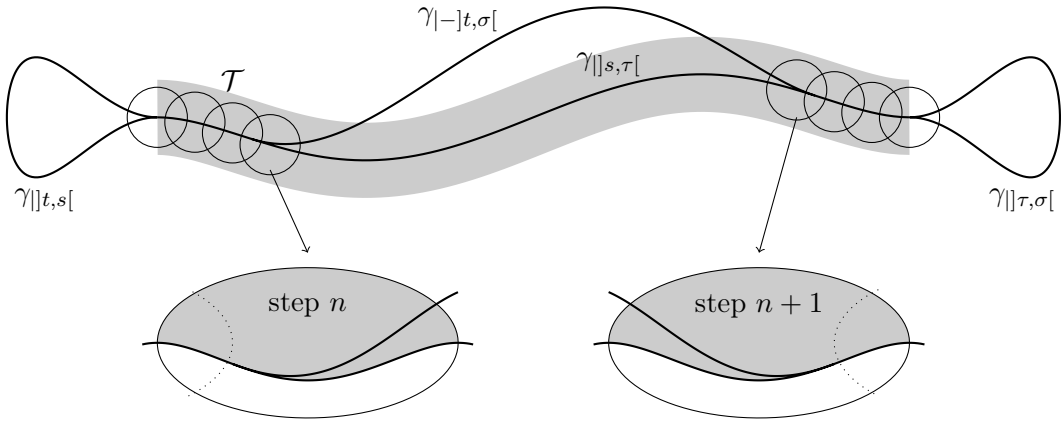
\end{proof}

In addition to the decomposition theorem (Theorem \ref{decomposition}), we can state the following result.

\begin{thm}\label{reparam}
Let $\gamma$ be a minimal curve for \eqref{problem}. There exists a parametrization of $\gamma$ such that
\[ \indc_\gamma\big( \RR\smallsetminus (\gamma) \big) \subset \{0,1\} . \]
Moreover, if $\OO$ is simply connected then
\[ \big\{ z \in \RR \ | \ \indc_\gamma(z) = 1 \big\} \subset \OO . \]
\end{thm}

\begin{proof}
Using Lemma \ref{parametr}, locally near a point where two branches of $\gamma$ meet and separate in the same side, we can define a small homotopic perturbation of order $\ee$ in the normal direction that splits them away. Doing this on the finite number of local neighborhoods mentioned above, we get a simple curve $\gamma_\ee$ very close to $\gamma$. Up to reversing the parametrization, we can suppose that $\indc_{\gamma_\ee}$ belongs to $\{0,1\}$. Given any $z \in \RR \smallsetminus (\gamma)$, by compactness $d(z,(\gamma))>0$ and then, for $\ee$ small enough, $\gamma_\ee$ is homotopically equivalent to $\gamma$ in $\RR\smallsetminus \{ z\}$. Then $\indc_\gamma(z) \in \{0,1\}$.

For the second proposition of the theorem, as $(\gamma) \subset \overline{\OO}$ and $\OO$ is simply connected, $\gamma$ cannot circle any point $z \in \RR\smallsetminus \OO$.
\end{proof}

\noindent We show in Figure \ref{generic} a few generic examples of curves which satisfy the conclusions of both Theorems~\ref{decomposition} and~\ref{reparam}, and therefore are candidates for being minimal curves (depending on the confinement)

\begin{figure}[!ht]
\begin{center}

\subfloat[no self-contact]{
\begin{tikzpicture}
\filldraw [draw=black,fill=gray!40]
	plot [domain=0:360,samples=100] ({\x}:{1.2+0.2*cos(3*\x)}) ;
\end{tikzpicture}
}
\qquad
\subfloat[side by side]{
\begin{tikzpicture}
\filldraw [draw=black,fill=gray!40]
	plot [domain=1:179,samples=100]
		({0.1-5*(\x/180)*(1-\x/180)},{2*sin(2*\x)*(\x/180)*(1-\x/180)})
	plot [domain=1:179,samples=100]
		({-0.1+5*(\x/180)*(1-\x/180)},{2*sin(2*\x)*(\x/180)*(1-\x/180)}) ;
\end{tikzpicture}
}
\qquad
\subfloat[surrounded]{
\begin{tikzpicture}
\filldraw [draw=black,fill=gray!40]
	plot [domain=0:180,samples=100]
		({4+6*(\x/180)*(1-\x/180)},{-2+sin(2*\x)*(\x/180)*(1-\x/180)})
	plot [domain=0:180,samples=100]
		({4-2*sin(3*\x)},{-2-4*sin(2*\x)*(\x/180)*(1-\x/180)}) ;
\end{tikzpicture}
}

\subfloat[two self-contacts]{
\begin{tikzpicture}
\filldraw [draw=black,fill=gray!40]
	plot [domain=-1:1]
		({\x},{(\x-1)^2*(\x+1)^2})
	plot [domain=-1:1]
		({\x},{-(\x-1)^2*(\x+1)^2})
	plot [domain=0:180,samples=100]
		({1+2*sin(\x)}, {-4*sin(2*\x)*(\x/180)*(1-\x/180)})
	plot [domain=0:180,samples=100]
		({-1-2*sin(\x)}, {-4*sin(2*\x)*(\x/180)*(1-\x/180)}) ;
\end{tikzpicture}
}
\qquad
\subfloat[several self-contacts]{
\begin{tikzpicture}
\filldraw [draw=black,fill=gray!40]
	plot [domain=-1:1]
		({\x},{35*(\x-1)^4*(\x+1)^4*(\x-1/2)^2*(\x+1/2)^2*(\x)^2})
	plot [domain=-1:1]
		({\x},{-35*(\x-1)^4*(\x+1)^4*(\x-1/2)^2*(\x+1/2)^2*(\x)^2})
	plot [domain=0:180,samples=100]
		({1+2*sin(\x)}, {-4*sin(2*\x)*(\x/180)*(1-\x/180)})
	plot [domain=0:180,samples=100]
		({-1-2*sin(\x)}, {-4*sin(2*\x)*(\x/180)*(1-\x/180)}) ;
\end{tikzpicture}
}

\subfloat[continuum of self-contacts]{
\begin{tikzpicture}
\filldraw [draw=black,fill=gray!40]
	(-1,0) -- (1,0)
	plot [domain=0:180,samples=100]
		({1+2*sin(\x)}, {-4*sin(2*\x)*(\x/180)*(1-\x/180)})
	plot [domain=0:180,samples=100]
		({-1-2*sin(\x)}, {-4*sin(2*\x)*(\x/180)*(1-\x/180)}) ;
\end{tikzpicture}
}

\subfloat[drop inside - drop outside]{
\begin{tikzpicture}
\filldraw [draw=black,fill=gray!40]
	plot [domain=-1:1,samples=150]
		(\x,{(\x-1)^2*(-0.25*\x-0.5)}) --
	plot [domain=0:180,samples=100]
		({1+3*(\x/180)*(1-\x/180)},{sin(2*\x)*(\x/180)*(1-\x/180)}) --
	(1,0) -- (-1,0) --
	plot [domain=0:180,samples=100]
		({-1-3*(\x/180)*(1-\x/180)},{-sin(2*\x)*(\x/180)*(1-\x/180)}) --
	plot [domain=0:1,samples=150] ({-1+sqrt(0.25-(\x-0.5)^2)},\x) --
	plot [domain=-1:1,samples=150] ({-1-sqrt(1-\x*\x)},-\x) --
	cycle ;
\end{tikzpicture}
}
\qquad
\subfloat[both drops outside]{
\begin{tikzpicture}
\filldraw [draw=black,fill=gray!40]
	plot [domain=0:1,samples=150] ({-1+sqrt(0.25-(\x-0.5)^2)},\x) --
	plot [domain=-1:1,samples=150] ({-1-sqrt(1-\x*\x)},-\x) --
	plot [domain=-1:1] (\x,{(\x-1)^2*(-0.25*\x-0.5)}) --
	plot [domain=0:180,samples=100]
		({1+3*(\x/180)*(1-\x/180)},{sin(2*\x)*(\x/180)*(1-\x/180)}) --
	plot [domain=-1:0,samples=150] ({1-sqrt(0.25-(\x+0.5)^2)},-1-\x) --
	plot [domain=-1:1,samples=150] ({1+sqrt(1-\x*\x)},\x) --
	plot [domain=-1:1,samples=100] (-\x,{(-\x+1)^2*(0.25*\x+0.5)}) --
	plot [domain=0:180,samples=100]
		({-1-3*(\x/180)*(1-\x/180)},{-sin(2*\x)*(\x/180)*(1-\x/180)}) --
	cycle ;
\end{tikzpicture}
}

\subfloat[inner empty drop]{
\begin{tikzpicture}
\filldraw [draw=black,fill=gray!40]
	plot [domain=0:180,samples=100]
		({2+3*(\x/180)*(1-\x/180)},{-sin(2*\x)*(\x/180)*(1-\x/180)}) --
	plot[domain=-1:1,samples=100] ({(3-\x)/2},{0.4*(1-\x)^2*(1+\x)^2}) --
	plot[domain=-1:1,samples=100] ({(-\x)^2+(3/2)*(-\x)-3/2},{0.8*(1-\x)^2*(1+\x)^2}) --
	plot[domain=-1:1,samples=100] ({(-3+\x)/2},{0.4*(1-\x)^2*(1+\x)^2}) --
	plot [domain=0:180,samples=100]
		({-1+3*(\x/180)*(1-\x/180)},{sin(2*\x)*(\x/180)*(1-\x/180)}) --
	plot[domain=-1:1,samples=100] ({(-3-\x)/2},{-0.4*(1-\x)^2*(1+\x)^2}) --
	plot[domain=-1:1,samples=100] ({(\x)^2+(3/2)*(\x)-3/2},{-0.8*(1-\x)^2*(1+\x)^2}) --
	plot[domain=-1:1,samples=100] ({(3+\x)/2},{-0.4*(1-\x)^2*(1+\x)^2}) --
	cycle ;
\end{tikzpicture}
}
\qquad
\subfloat[inner full drop]{
\begin{tikzpicture}
\filldraw [draw=black,fill=gray!40]
	plot [domain=0:180,samples=100]
		({1+3*(\x/180)*(1-\x/180)},{-sin(2*\x)*(\x/180)*(1-\x/180)}) --
	plot[domain=-1:1,samples=100] ({(-\x)^2+(3/2)*(-\x)-3/2},{0.8*(1-\x)^2*(1+\x)^2}) --
	plot[domain=-1:1,samples=100] ({-3/2-(1/2)*((-\x)^2+(3/2)*(-\x)-3/2)},{0.4*(1-\x)^2*(1+\x)^2}) --
	plot [domain=0:180,samples=100]
		({-1/2-3*(\x/180)*(1-\x/180)},{sin(2*\x)*(\x/180)*(1-\x/180)}) --
	plot[domain=-1:1,samples=100] ({-3/2-(1/2)*((\x)^2+(3/2)*(\x)-3/2)},{-0.4*(1-\x)^2*(1+\x)^2}) --
	plot[domain=-1:1,samples=100] ({(\x)^2+(3/2)*(\x)-3/2},{-0.8*(1-\x)^2*(1+\x)^2}) --
	cycle ;
\end{tikzpicture}
}

\end{center}
\caption{Generic curves satisfying the conclusions of Theorems~\ref{decomposition} and~\ref{reparam}}
\label{generic}
\end{figure}
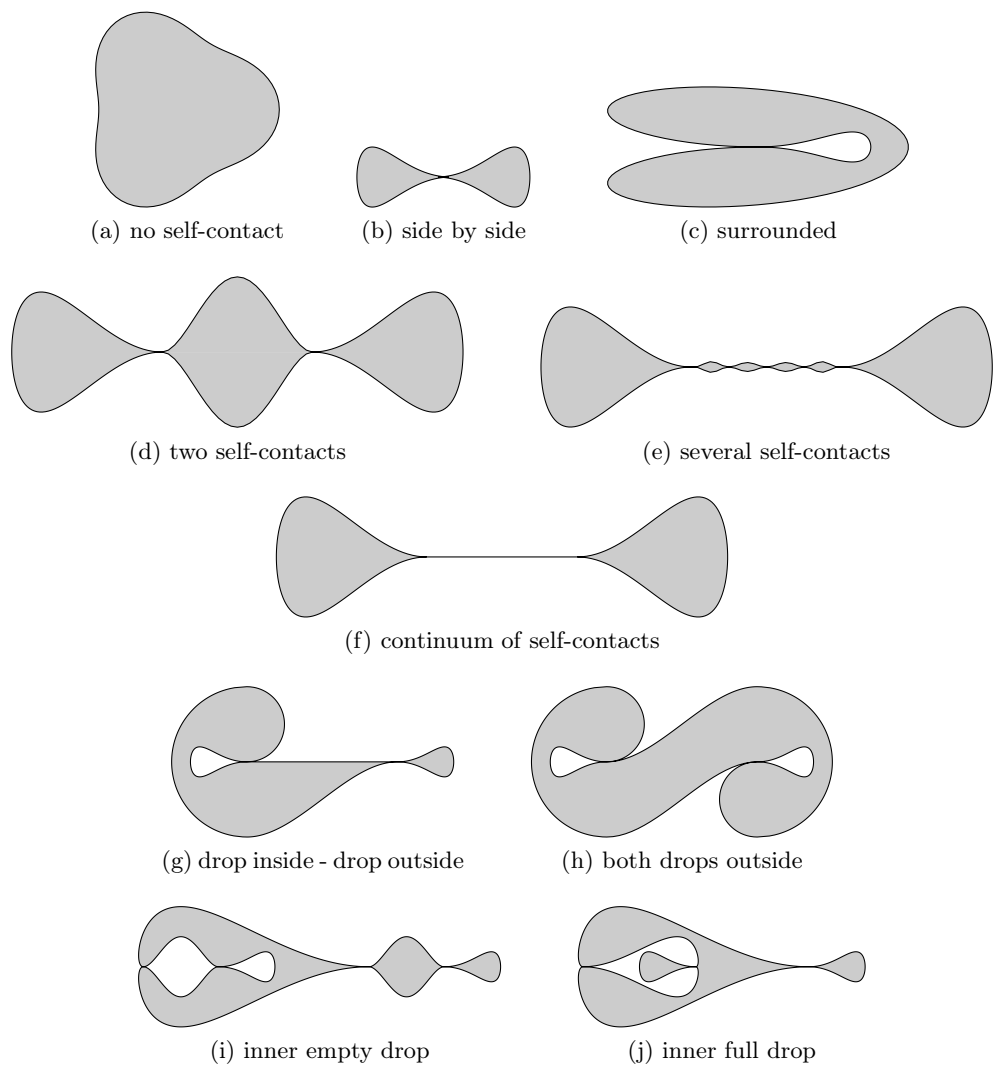


\section{Regularity} \label{regularite}

This section is devoted to the calculation of the Euler equation of the energy $\WW$, and to the regularity of minimal curves of \eqref{problem}. In particular, we prove carefully that if $\gamma$ is a minimal curve then $d\WW(\gamma)$ is a Radon measure, from which bootstrap arguments yield more regularity, out of contact points.

\begin{propo}[First variation] \label{variation premiere}
Let  $\gamma : [0,L] \to \RR $ be a regular curve with arc length parametrization, for all $\delta \in C^{\infty}([0,L],\RR)$ and for $\ee>0$, we have
\[ \WW(\gamma+\ee\delta) = \WW(\gamma) +
\ee \int_0^L \left( 2\gamma''\cdot \delta'' - 3 (\gamma'\cdot \delta')|\gamma''|^2 \right) \ud s
+o(\ee) . \]
\end{propo}

\begin{proof}
It is a classical first order expansion calculation, second order terms in $\ee$ remain bounded thanks to $||\gamma''||_{L^2} < +\infty $.
\end{proof}

\begin{notation}
Using the weak derivatives, we can write the first variation in distributional sense: $\forall \delta \in C^{\infty}([0,L],\RR)$,
\[ \ud \WW(\gamma).\delta = 
\int_0^L \left( 2\gamma''\cdot \delta'' - 3 (\gamma'\cdot \delta')|\gamma''|^2 \right) \ud s
= \int_0^L \left( 2\gamma^{[4]}+ 3 (|\gamma''|^2 \gamma')' \right) \cdot \delta \ud s .
\]
\end{notation}

\noindent Let us first sketch how we will study the regularity of minimal curves: without restrictions on the perturbations, we get from classical regularity theory that the minimality of $\gamma$ implies that
\[ \ud \WW(\gamma).\delta \geqslant 0 \quad \text{for all perturbations } \delta \in C^{\infty}([0,L],\RR) \]
from which we deduce that $\ud \WW(\gamma) = 0$. Then, a bootstrap argument yields that $\gamma$ is smooth. This argument, however, is valid only locally on the free parts of $\gamma$, i.e. out of contact or self-contact points. For self-intersections and contacts with $\od \OO$, we can allow local perturbations only in one direction. Then we get that $\ud \WW(\gamma)$ is a measure, and this yields
higher regularity than $H^2$. In the following, we describe these arguments with more details in a local neighborhood $\mathcal{U}$ of an arbitrary point of a minimal curve $\gamma$.

\subsection{Admissible perturbations}

Let $\gamma$ be a minimal curve and $\delta \in C^{\infty}(\cer,\RR)$ a perturbation. We have to deal with the contact with $\od \OO$ and self-intersections. We say that $\delta$ is an admissible perturbation if, for $\ee>0$ small enough, $\gamma+\ee\delta \in \adm$. As we have two constraints in $\adm$ (confinement in $\OO$ and tangential self-intersections), we will sometimes refer to only one of these constraints.

Actually, we only have to take care of self-intersections around the particular self-contact couples given by Theorem \ref{decomposition} (see the definition below). Indeed, suppose $\gamma_{|]s,\tau[}$ and $\gamma_{|]\sigma,t[}$ are not empty, according to Theorem \ref{decomposition} they have the same energy. If $\delta$ is a perturbation of $\gamma_{|]s,\tau[}$ admissible with respect to the confinement leading to a decrease of energy:
\[ \WW(\gamma_{|]s,\tau[} + \delta) < \WW(\gamma_{|]s,\tau[}) , \]
then we can replace $\gamma_{|]\sigma,t[}$ by $\gamma_{-|]s,\tau[}$ and expand the perturbation to $\gamma_{|]\sigma,t[}$ by taking $\delta(t)=\delta(s)$ for all $t,s$ such that $\gamma(t)=\gamma(s)$. Denoting $\tilde{\gamma}$ the new curve and $\tilde{\delta}$ the new perturbation, $\tilde{\delta}$ is a admissible perturbation for $\tilde{\gamma}$ (with respect to both constraints). Then we obtain an admissible curve with strictly less energy than $\gamma$, which is impossible. Therefore, the perturbations must be restricted only along the confinement boundary and near the particular self-contact couples (see below).

\begin{Def}[Particular self-contact points]
The particular self-contact points of a minimal curve $\gamma\in\adm$ are:
\begin{itemize}
\item the unique self contact point of $\gamma$ if it exists,
\item both self-contact points $(t,s)$ and $(\tau,\sigma)$ given by Theorem \ref{decomposition} if $\gamma$ admits at least two distinct self-contact points.
\end{itemize}
Obviously, a simple curve has no particular self-contact point.
\end{Def}

Let $\nu$ denote a smooth vector field parametrized on $\cer$ and supported on $\gamma$ (i.e. $\nu(s)=\tilde\nu(\gamma(s))$), and such that $\nu$ vanishes away from $\mathcal{U}$. We take $\nu=0$ locally around a free part (i.e. no contact with $\od \OO$ and far away from particular self-contact points). Otherwise, we use the parametrization of Lemma \ref{parametr} and we write both branches of $\gamma$ (recall from Theorem~\ref{decomposition} that there cannot be more than two branches) and $\od \OO\cap{\mathcal U}$ as graphs of functions in $\mathcal{U}$. Thanks to this parametrization, we have that both branches and $\od \OO$ are ordered in $\mathcal{U}$.
If we have only a particular self-couple, we choose $\nu$ so that it heads north on the above branch and it heads south on the below branch (see Figure \ref{perturbation 1}). In addition, we assume without loss of generality that $|\nu|=1$ strictly within $\mathcal{U}$. If locally $\gamma$ makes contacts only with $\od \OO$ then $\nu$ is built so that it heads toward the inner $\od \OO$ (see Figure \ref{perturbation 2}). Again, we let $|\nu|=1$ strictly within $\mathcal{U}$. The last case is when we have both a self-contact and a contact with $\Omega$, and without loss of generality we may assume that $\od\OO$ is below both branches of $\gamma$. In such case, we choose $\nu$ so that it heads north both for the top and above branches of $\gamma$. Then we set, strictly within $\mathcal{U}$,  $|\nu|=2$ on the above branch and $|\nu|=1$ on the below branch, and we impose $|\nu_{below}| \leqslant |\nu_{above}|$ in $\mathcal{U}$ (see Figure \ref{perturbation 3}) where, denoting as $\gamma_{below}$ and $\gamma_{above}$ both branches of $\gamma$ and considering $\tilde\nu$ as before, we define $\nu_{below/above}=\tilde\nu(\gamma_{below/above})$.

\begin{figure}[!ht]
\begin{center}
\subfloat[self-contact \label{perturbation 1}]{
\begin{tikzpicture}
\draw [thick] plot [domain=-2:2, samples=100] (\x,{2.5+(1/9)*(\x*\x)}) ;
\draw [thick] plot [domain=-2:2, samples=100] (\x,{2.5-(2/9)*(\x*\x)}) ;
\foreach \t in {-1.8,-1.6,...,1.8}
	{\draw [->] (\t,{2.5+(1/9)*(\t*\t)}) -- (\t,{0.5+2.5+(1/9)*(\t*\t)}) ; }
\foreach \t in {-1.8,-1.6,...,1.8}
	{\draw [->] (\t,{2.5-(2/9)*(\t*\t)}) -- (\t,{-0.5+2.5-(2/9)*(\t*\t)}) ; }
\draw (2.5,3) node{$\gamma_{above}$}
	(2.5,1.5) node{$\gamma_{below}$} ;
\end{tikzpicture}
}
~
\subfloat[contact with the boundary \label{perturbation 2}]{
\begin{tikzpicture}
\draw [thick] plot [domain=-2:2, samples=100] (\x,{2.5+(1/9)*(\x*\x)}) ;
\draw [thick] plot [domain=-2:2, samples=100] (\x,{2.5-(1/9)*(\x*\x)}) ;
\foreach \t in {-1.8,-1.6,...,1.8}
	{\draw [->] (\t,{2.5+(1/9)*(\t*\t)}) -- (\t,{0.5+2.5+(1/9)*(\t*\t)}) ; }
\draw (2.5,3) node{$\gamma$}
	(2.5,2) node{$\od \OO$} ;
\end{tikzpicture}
}
~
\subfloat[both contacts \label{perturbation 3}]{
\begin{tikzpicture}
\draw [thick] plot [domain=-2:2, samples=100] (\x,{2.5+(1/9)*(\x*\x)}) ;
\draw [thick] plot [domain=-2:2, samples=100] (\x,{2.5-(1/9)*(\x*\x)}) ;
\draw [thick] plot [domain=-2:2, samples=100] (\x,{2.5-(2/9)*(\x*\x)}) ;
\foreach \t in {-1.8,-1.6,...,1.8}
	{\draw [->] (\t,{2.5+(1/9)*(\t*\t)}) -- (\t,{1+2.5+(1/9)*(\t*\t)}) ; }
\foreach \t in {-1.7,-1.5,...,1.7}
	{\draw [->] (\t,{2.5-(1/9)*(\t*\t)}) -- (\t,{0.5+2.5-(1/9)*(\t*\t)}) ; }
\draw (2.5,3) node{$\gamma_{above}$}
	(2.5,2) node{$\gamma_{below}$}
	(2.5,1.5) node{$\od \OO$} ;
\end{tikzpicture}
}
\end{center}
\caption{Direction of the perturbation}
\label{vec}
\end{figure}
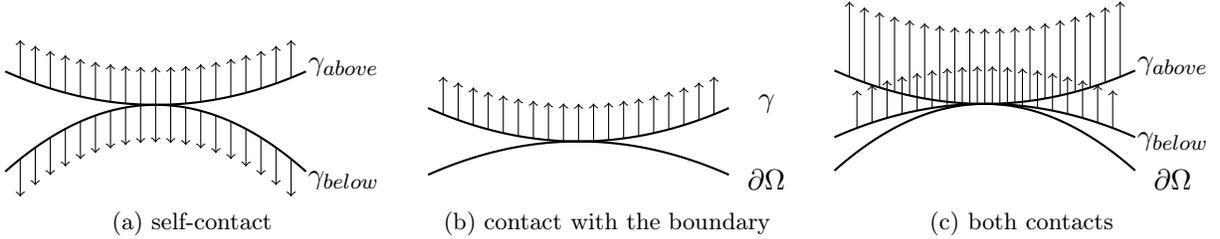

Once $\nu$ is defined as above, for all perturbations $\delta \in C^{\infty}(\cer,\RR)$, we define
\[ \delta_1 = \delta + 2||\delta||_{\infty}\nu \]
and \[ \delta_2 = -\delta + 2||\delta||_{\infty}\nu . \]
The construction of $\nu$ forces $\delta_1$ and $\delta_2$ to be admissible and then
\[ \ud \WW(\gamma).\delta_1 \geqslant 0 \quad \text{and} \quad \ud \WW(\gamma).\delta_2 \geqslant 0 .\]
Finally
\[ |\ud \WW(\gamma).\delta| \leqslant
\big( 2 |\ud \WW(\gamma).\nu|\big) ||\delta||_{\infty} . \]
and the following proposition is proved:
\begin{propo}\label{proBV}
Let $\gamma$ be a minimal curve for~\eqref{problem}, then $\ud \WW(\gamma)$ is a $\RR-$valued Radon measure.
\end{propo}

\begin{rmq}
On free parts (away from contact with $\od \OO$ and particular self-contact points), we have $\nu = 0$ and then, for any local perturbation $\delta$, $\ud \WW(\gamma).\delta = 0$, thus $\ud \WW(\gamma)=0$.
\end{rmq}

\subsection{Regularity improvement}

By definition, a minimal curve $\gamma$ belongs to $H^2$. However, as $\ud \WW(\gamma) = 2\gamma^{[4]}+ 3 (|\gamma''|^2 \gamma')'$ is a measure, we will prove that $\gamma \in W^{3,\infty}$ using the theory of $BV$ functions (see \cite{Ambrosio2000}). We first recall  a classical result in functional analysis:

\begin{lem} \label{derive}
If $u,v \in BV(\cer,\RR)$ and $\lambda \in BV(\cer,\mathbb{R})$ then $u\cdot v$ and $\lambda u$ belong to $BV(\cer,\mathbb{R})$. Moreover, if $u,v \in W^{1,\infty}(\cer,\RR)$ then $(uv)' \in L^{\infty}$. In both case, the Leibniz's rule is valid.
\end{lem}

\begin{proof}
Combine the two following results on the components of $u$ and $v$:
\begin {itemize}
	\item
		If $u,v \in BV([a,b],\mathbb{R})$ then $uv$ is a function of bounded variation and, up to a choice of an appropriate representative, in the measure sense
		\[ (uv)'=u'v+uv' . \]
		See \cite[ex 3.97]{Ambrosio2000}.
	\item
		Let $p\in[1,+\infty]$, if $u,v \in L^{\infty}\cap W^{1,p}([a,b],\mathbb{R})$ then $uv$ is a function of $W^{1,p}$ and
		\[ (uv)'=u'v+uv' . \]
		See \cite[4.2.2 - Theorem 4]{Evans1992}.
\end{itemize}
\end{proof}

\noindent We now state the main result of this section, that is the regularity of minimizers.

\begin{thm}
Let $\gamma$ be a minimal curve, we have
\[ \gamma \in W^{3,\infty}(\cer,\RR), \]
\[ \gamma^{[3]} \in BV(\cer,\RR) \]
and \[ \kappa' \in BV(\cer,\mathbb{R}) . \]
Moreover, $\gamma$ is $C^\infty$ on every open interval of $\cer$ away from particular self-contact couples and from contact points with $\od \OO$.
\end{thm}

\begin{proof}
The result follows from the fact that 
$2\gamma^{[3]}+ 3 (|\gamma''|^2 \gamma')$ is a $BV$ function
thanks to Proposition \ref{proBV}. 
Then, we use Lemma \ref{derive} and the following embeddings for functions $\cer \to \RR$:
\[ W^{2,\infty} \hookrightarrow W^{1,\infty} \hookrightarrow W^{1,1} \hookrightarrow BV \hookrightarrow L^{\infty} \hookrightarrow L^1 . \]

For the curvature, using arc length parametrization, we have $\kappa = \gamma''\cdot (\gamma')^{\perp}$ (where $(\gamma')^{\perp}$ stands for the unit orthogonal vector to $\gamma'$ in the counterclokwise sense) and then, applying Lemma \ref{derive}, we get
\[ \kappa' = \gamma^{[3]}\cdot (\gamma')^{\perp} + \underbrace{\gamma''\cdot (\gamma'')^{\perp}}_{=0} \in BV. \]

For the free part, we observe that $2\gamma^{[3]}+ 3 (|\gamma''|^2 \gamma')$ is a constant, and since $W^{3,\infty} \hookrightarrow C^{2,\alpha}$ we derive the smoothness using a bootstrap argument.
\end{proof}

\begin{rmq} \
\begin{itemize}
\item The embedding $W^{3,\infty}(\cer,\RR) \hookrightarrow C^{2,\alpha}(\cer,\RR)$ for some $\alpha$ guarantees that any minimal curve is at least of class $C^2$.
\item Locally around points where $\gamma$ is smooth (at least $C^4$) we can recover the classical Euler-Lagrange equation for the curvature
\[ \ud \WW(\gamma) = 2\gamma^{[4]}+ 3 (|\gamma''|^2 \gamma')' = \Big( 2\kappa''+\kappa^3 \Big) \vec{n} \]
with $\vec{n}=(\gamma')^{\perp}$ the unit normal vector to $\gamma$.
\end{itemize}
\end{rmq}


\section{Convexity} \label{convexite}

In this section, we show that if the constraint $\OO$ is convex then every optimal curve circles a convex set. We consider $\gamma$ a minimal curve for \eqref{problem} with $\OO$ convex. We shall use a rigourous proof, but the conclusion can also be understood with unformal arguments:  if we consider a non convex curve $\gamma$ and replace it with its convex envelop, this amounts to replacing some pieces of $\gamma$ by straight segments without creating angles. As a straight segment has zero energy and the removed parts have non zero energy, it follows that the convex envelope has strictly lesser energy. This argument is, however, difficult to write rigourously. We rather use a proof by contradiction based on the support straight lines of $\gamma$.

\begin{Def}[Support straight line]
Let $A$ be a subset of $\RR$. We say that $A$ admits a support straight line at $x \in \od A$ if there exists a linear form $\Lambda$ and a constant $\lambda$ such that
\[ x \in D=\{ y\in \RR \ | \ \Lambda(y)=\lambda \} \]
and
\[ A \subset \{ y\in \RR \ | \ \Lambda(y) \geqslant \lambda \} . \]
\end{Def}

\noindent Geometrically, $D$ is a support straight line of $A$ if $D \cap \od A \neq \emptyset$ and $A$ belongs to only one side of $D$. We will use the following characterisation of convex sets: a (closed with non-empty interior) subset $A$ is convex if and only if every point of its boundary admits a support straight line (see \cite[Theorem 2.7 (iii)]{Giaquinta2012}). In our case, $A$ will be the (closed) set surrounded by $(\gamma)$ and, as the curve is $C^1$, a support straight line is a tangent straight line:
\[ D_t = \gamma(t) + \mathbb{R} \gamma'(t) \]
We will use Lemma \ref{parametr} and the following three lemmas to show that every tangent straight line is a support straight line. For the rest of the section we will use the parametrization of Lemma \ref{parametr}.

\begin{lem} \label{parallele}
For every straight line $\Delta$ of $\RR$, there exist two support straight lines of $(\gamma)$ parallel to $\Delta$.
\end{lem}
This lemma implies that there are infinitely  many support straight lines. The proof requires only the boundedness of the set $(\gamma)$.

\begin{proof}
Up to a rotation and a translation, we can suppose that $\Delta$ is the horizontal axis $\{y=0\}$ of $\RR$. As $\gamma$ is bounded, there exist two straight lines $\Delta_a=\{y=a\}$ and $\Delta_b=\{y=b\}$ (with $b<0<a$) parallel to $\Delta$ such that $\gamma$ belongs to the domain $\{ (x,y)\in \RR \ | \ a\leqslant y \leqslant b \} $. Among them, consider the closest to $(\gamma)$:
\[ a^+ = \inf \{ a>0 \ | \ (\gamma) \text{ is below the straight line } \Delta_a : y=a \} \]
and \[ b^- = \sup \{ b<0 \ | \ (\gamma) \text{ is above the straight line } \Delta_b : y=b \} . \]
Therefore $\Delta^+ : y=a^+$ and $\Delta^- : y=b^-$ are distinct support straight lines of $(\gamma)$ parallel to $\Delta$.
\end{proof}

\begin{lem} \label{segment}
$I=\{ t\in\cer \ | \ D_t=\gamma(t) + \mathbb{R} \gamma'(t) \text{ is a support straight line} \}$ is closed.
\end{lem}
This lemma remains valid for every $C^1$ curve.

\begin{proof}
Consider $t\in \cer \smallsetminus I$, $D_t$ is not a support straight line so there exist two distinct $s_1$ and $s_2$ in $\cer$ different from $t$ belonging respectively (and strictly) to one side and to the other side of $D_t$. $t\mapsto D_t$ is continuous (because $\gamma$ is $C^1$) and therefore there exists $\ee>0$ such that for all $\eta \in ]-\ee,\ee[$, $D_{t+\eta}$ separates strictly $\gamma(s_1)$ and $\gamma(s_2)$. Thus $D_{t+\eta}$ is not a support straight line and therefore $]t-\ee,t+\ee[ \subset \cer \smallsetminus I$. Thus $I$ is closed.
\end{proof}

\begin{lem} \label{support}
For all $t$ such that $D_t$ is a support straight line, $I_t = \{ s\in\cer \ | \ \gamma(s)\in D_t \}$ is a segment of $\cer$.
\end{lem}

\begin{proof}
The proof uses the minimality of $\gamma$.
As $\gamma$ is continuous and $D_t$ is closed, we only have to prove that $I_t$ is connected.
Consider $t$ such that $I_t$ is disconnected and $D_t$ is a support straight line. There exists $s\in I_t$ such that $\gamma$ does not go along $[\gamma(t),\gamma(s)]$ (as a segment of $\RR$). Since $D_t$ is a support straight line, $\gamma$ is tangent to $D_t$ at every point of $I_t$. Therefore $\gamma'(t)$ and $\gamma'(s)$ belongs to $\overrightarrow{D}_t$. We denote by $[\gamma(t),\gamma(s)]$ the oriented segment of $\RR$ linking $\gamma(t)$ toward $\gamma(s)$.

If $\gamma'(t) = \gamma'(s)$, up to switching $\gamma(t)$ and $\gamma(s)$ we can suppose $\overrightarrow{\gamma(t)\gamma(s)}\cdot\gamma'(t)>0$ (see Figure \ref{convex}), then consider $\tilde{\gamma} = [\gamma(t),\gamma(s)] \oplus \gamma_{|]s,t[}$. As every contact between $\gamma$ and $D_t$ are tangent and $(\gamma)$ is not included in $D_t$:
\[ \int_{([\gamma(t),\gamma(s)])} \kappa^2 \ud \mathcal{H}^1 \quad = \quad 0 \quad < \quad \int_{(\gamma_{|[t,s]})} \kappa^2 \ud \mathcal{H}^1 . \]
Then $\WW(\tilde{\gamma}) < \WW(\gamma)$ and this is in contradiction with the optimality.

If $\gamma'(t) =- \gamma'(s)$ then one part of $\gamma$ (near $t$ or $s$) is coming in the domain delimited by $\gamma_{[t,s]}$ and $D_t$, and the other part is going out of the domain. By continuity, there exist $u_1,u_2 \in \cer$ such that $\gamma(u_1)=\gamma(u_2)$ and $t<u_1<s<u_2$ (see Figure \ref{convex}). As before, we can suppose that $\gamma'(s)\cdot\overrightarrow{\gamma(s)\gamma(t)}>0$. Consider $\tilde{\gamma} = \gamma_{|[t,u_1]} \oplus \gamma_{-|[u_2,s]} \oplus [\gamma(s),\gamma(t)]$. As before, $\tilde{\gamma}$ has less energy than $\gamma$ which is impossible.

\begin{figure}[!ht]
\begin{center}

\begin{tikzpicture}
\draw [thick] (2,-1)+(60:1) arc (60:270:1) --
	plot [domain=-1:1.3] (3+\x,{2*(\x-1)^2*(-0.25*\x-0.5)}) ;
\draw [thick,dotted]
	(2,-1)+(30:1) arc (30:60:1)
	plot [domain=1.3:1.6] (3+\x,{2*(\x-1)^2*(-0.25*\x-0.5)}) ;
\draw (1,0) -- (5,0) ;
\draw [->] (1.8,0.2) -- (2.2,0.2) ;
\draw [->] (4.2,0.2) -- (3.8,0.2) ;
\draw (2,-0.3) node{$\gamma(s)$} (4,-0.3) node{$\gamma(t)$} ;
\draw (5.5,0) node{$D_t$} ;
\draw (3,1) node{$\gamma'(s)=-\gamma'(t)$} ;

\draw (-5,0) -- (-1,0) ;
\draw [thick] plot [domain=-1:1] ({\x-3},{-(\x-1)^2*(\x+1)^2}) ;
\draw [thick]
	plot [domain=-1.5:-1] ({\x-3},{-(\x+1)^2})
	plot [domain=1:1.5] ({\x-3},{-(\x-1)^2}) ;
\draw [thick,dotted]
	plot [domain=-1.8:-1.5] ({\x-3},{-(\x+1)^2})
	plot [domain=1.5:1.8] ({\x-3},{-(\x-1)^2}) ;
\draw [->] (-4.2,0.2) -- (-3.8,0.2) ;
\draw [->] (-2.2,0.2) -- (-1.8,0.2) ;
\draw (-4,-0.5) node{$\gamma(t)$} (-2,-0.5) node{$\gamma(s)$} (-5.5,0) node{$D_t$} ;
\draw (-3,1) node{$\gamma'(s)=\gamma'(t)$} ;
\end{tikzpicture}

\end{center}
\caption{Two cases of tangent contact with $D_t$}
\label{convex}
\end{figure}
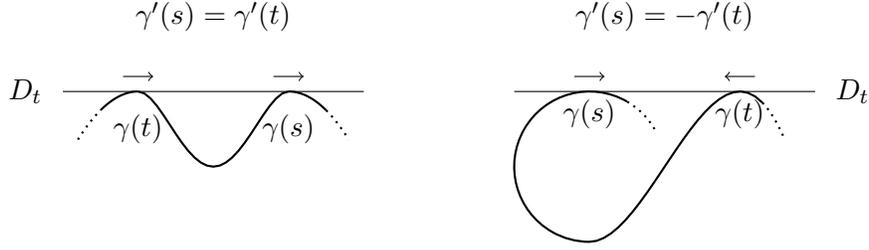
Thus for all $t$ such that $D_t$ is a support straight line, $I_t$ is a segment of $\cer$.
\end{proof}

Now we will prove the main result of this section.

\begin{thm}
If $\OO$ is convex then every optimal curve surrounds a convex set.
\end{thm}

\begin{proof}
Recall that $D_t = \gamma(t) + \mathbb{R} \gamma'(t)$ and $I=\{ t\in\cer \ | \ D_t \text{ is a support straight line} \}$ is closed thanks to Lemma \ref{segment}. Assume there exists $t$ in $\cer \smallsetminus I$ and consider $]a,b[$ the maximal open interval of $\cer$ containing $t$ such that
$]a,b[ \subset \cer \smallsetminus I$. As $\cer \smallsetminus I$ is open, the maximality implies that $D_a$ and $D_b$ are support straight lines of $(\gamma)$.

\textbf{Case 1:} $D_a=D_b$. \\
In that case, $\gamma(a),\gamma(b) \in D_b$, and using Lemma \ref{support} we obtain $\gamma_{-|[a,b]} \subset D_a$. Thanks to Lemma \ref{parallele}, there exists a support straight line $D$ different from $D_a=D_b$. If there exists a point $s$ such that $D=D_s$, then necessarily $s\in ]a,b[$ because the whole branch $\gamma_{-|[a,b]} \subset D_a$. This is in contradiction with $]a,b[\subset \cer \smallsetminus I$, see Figure \ref{cas1-2}.

\textbf{Case 2:} $D_a$ and $D_b$ are distinct and parallel. \\
In that case, up to using a rotation we can suppose that $D_a$ and $D_b$ are horizontal lines. $(\gamma)$ is confined between $D_a$ and $D_b$. Using Lemma \ref{parallele}, there exist two vertical support straight lines, one on the right and one on the left of $(\gamma)$. One of these is tangent to the branch $\gamma_{|]a,b[}$. Indeed suppose that $\gamma_{-|[a,b]}$ links both vertical straight lines. Inside the rectangle defined by both vertical and both horizontal straight lines, we have that $\gamma_{-|[a,b]}$ links the vertical edges and $\gamma_{|[a,b]}$ links the horizontal edges. These edges are parts of support straight lines therefore each contact with $\gamma$ is not localised on the vertices of the rectangle. Thus $\gamma_{-|[a,b]}$ must cross $\gamma_{|]a,b[}$ and that is impossible with the parametrization of Lemma \ref{parametr}. Consequently the branch $\gamma_{|]a,b[}$ admits a support straight line which is impossible since $]a,b[\subset \cer \smallsetminus I$. Therefore, this is again an impossible case (see Figure \ref{cas1-2}).

\begin{figure}[!ht]
\begin{center}

\begin{tikzpicture}
\draw [thick] plot [domain=-1:1, samples=100] ({-5+cos(135*\x+90)},{(\x-1)^2*(\x+1)^2}) -- cycle ;
\draw (-7,0) -- (-3,0) (-5.7,0) node{$|$} (-4.3,0) node{$|$} ;
\draw (-5,-0.3) node{$\gamma_{-|[a,b]}$} (-5,1.3) node{$\gamma_{|[a,b]}$} ;
\draw (-4.4,0.4) node{$\gamma(a)$} (-5.6,0.4) node{$\gamma(b)$} ;
\draw (-7,-0.3) node{$D_a=D_b$} ;
\draw (-5,-1) node{confounded} ;

\draw [thick] plot [domain=-1:1, samples=100] ({\x},{1-0.25*(\x+1)^2*(2-\x)}) ;
\draw [thick] (-1,0.75)+(90:0.25) arc (90:280:0.25) ;
\draw [thick] (1,0.25)+(270:0.25) arc (270:460:0.25) ;
\draw (-2,0) -- (2,0) (-2,1) -- (2,1) ;
\draw (-1.25,1.5) -- (-1.25,-0.5) ;
\draw (1.25,1.5) -- (1.25,-0.5) ;
\draw [thick,dashed] (-1,0.5) -- (-0.5,0.5) (1,0.5) -- (0.5,0.5) ;
\draw (-0.8,1.4) node{$\gamma(a)$} (0.8,-0.4) node{$\gamma(b)$}
(0.7,1.2) node{$D_a$} (-0.7,-0.2) node{$D_b$} ;
\draw (-1,1) node{$|$} (1,0) node{$|$} ;
\draw (0,-1) node{non-confounded} (0,-1.4) node{impossible} ;

\draw [thick]
	plot [domain=-1:1, samples=100]
		({5-0.93*\x*\x+\x+0.93},{1-0.25*(\x+1)^2*(2-\x)}) --
	plot [domain=-1:1, samples=100]
		({5+0.93*(1-\x)^2-(1-\x)-0.93},{0.25*(\x+1)^2*(2-\x)}) ;
\draw (3,0) -- (7,0) (3,1) -- (7,1) ;
\draw (3.8,1.5) -- (3.8,-0.5) ;
\draw (6.2,1.5) -- (6.2,-0.5) ;
\draw (4.2,1.4) node{$\gamma(a)$} (5.8,-0.4) node{$\gamma(b)$}
	(5.7,1.2) node{$D_a$} (4.3,-0.2) node{$D_b$}
	(3,0.3) node{$\gamma_{|[a,b]}$} (7,0.7) node{$\gamma_{-|[a,b]}$} ;
\draw[->] (3.4,0.3) -- +(0.6,0.1) ;
\draw[->] (6.5,0.7) -- +(-0.6,-0.1) ;
\draw (4,1) node{$|$} (6,0) node{$|$} ;
\draw (5,-1) node{non-confounded} ;
\end{tikzpicture}

\end{center}
\caption{Parallel case}
\label{cas1-2}
\end{figure}
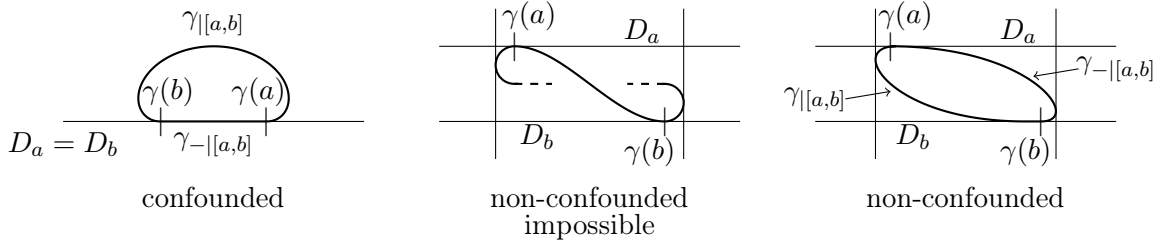

\textbf{Case 3:} $D_a$ and $D_b$ are not parallel. \\
In that case, we denote as $z$ the intersection point. Up to using a rotation, we suppose that the straight line $(\gamma(a)\gamma(b))$ is horizontal and $z$ is above this line. Consider the domain $U$ delimited by the segments (of $\RR$) $[z,\gamma(a)]$ and $[z,\gamma(b)]$ and by $\gamma_{|[a,b]}$. Using again Lemma \ref{parametr}, $\gamma_{-|[a,b]}$ is either included in the closed domain $U$ (see Figure \ref{cas3}, inner case), either out of the open domain $U$ (see Figure \ref{cas3}, outer case). For the inner case, there exists a horizontal support straight line below $(\gamma)$. The corresponding tangent straight line $D_s$ satisfies $s\in ]a,b[$ which is impossible. For the outer case, the argument is similar using a horizontal support straight line above $(\gamma)$.

\begin{figure}[!ht]
\begin{center}

\begin{tikzpicture}
\fill [color=gray!40] (0,1.5) -- (-2,0) --
	plot [domain=-1.8:0,samples=150] ({-1.4-sqrt(1-(-1.8-\x+0.8)^2)},{-1.8-\x}) --
	plot [domain=0:1] ({-1.4+3.1*\x},{-1.8+0.9*(\x)^2*(3-2*\x)}) --
	plot [domain=-0.9:0,samples=100] ({1.7+sqrt(0.25-(\x+0.4)^2)},{\x}) --
	(2,0) -- cycle ;
\draw (-3.5,-1.125) -- (0.33,1.75)
	(3.5,-1.125) -- (-0.33,1.75)
	(-3,-1.8) -- (3,-1.8) ;
\draw [thick]
	(-1.4,-0.8)+(90:1) arc (90:130:1)
	plot [domain=-1.8:0,samples=150] ({-1.4-sqrt(1-(-1.8-\x+0.8)^2)},{-1.8-\x}) --
	plot [domain=0:1] ({-1.4+3.1*\x},{-1.8+0.9*(\x)^2*(3-2*\x)}) --
	plot [domain=-0.9:0,samples=100] ({1.7+sqrt(0.25-(\x+0.4)^2)},{\x})
	(1.7,-0.4)+(50:0.5) arc (50:90:0.5) ;
\draw [thick, dotted] (-1,0.2) -- (-1.4,0.2)  (1.3,0.1) -- (1.7,0.1) ;
\draw (-2.1,0.1) -- (-1.9,-0.1) (1.9,-0.1) -- (2.1,0.1) ;
\draw (0,1.5) node{$\bullet$} (0,1.8) node{$z$} (0,0) node{$U$} 
	(-2.2,0.3) node{$\gamma(a)$} (2.2,0.3) node{$\gamma(b)$}
	(1.5,-1.2) node{$\gamma_{|]a,b[}$} (-3.4,-1.8) node{$D_s$} 
	(0,-2.5) node{inner case} ;

\fill [color=gray!40] (8,1.5) -- (6,0) --
	plot [domain=-2:2] ({\x+8},{-(3/16)*((\x)^2-4)}) --
		(10,0) -- cycle ;
\draw (4.5,-1.125) -- (8.33,1.75)
	(11.5,-1.125) -- (7.66,1.75)
	(5,0.75) -- (11,0.75) ;
\draw [thick]
	(6.6,-0.8)+(125:1) arc (125:170:1)
	plot [domain=-2:2] ({\x+8},{-(3/16)*((\x)^2-4)}) 
		(9.7,-0.4)+(0:0.5) arc (0:55:0.5) ;
\draw [thick,dotted]
	(6.6,-0.8)+(170:1) arc (170:200:1)
	(9.7,-0.4)+(300:0.5) arc (300:360:0.5) ;
\draw (5.9,0.1) -- (6.1,-0.1) (9.9,-0.1) -- (10.1,0.1) ;
\draw (8,1.5) node{$\bullet$} (8,1.8) node{$z$} (8,1.1) node{$U$}
	(5.8,0.3) node{$\gamma(a)$} (10.2,0.3) node{$\gamma(b)$}
	(9,0.1) node{$\gamma_{|]a,b[}$} (4.6,0.75) node{$D_s$}
	(8,-2.5) node{outer case} ;
\end{tikzpicture}

\end{center}
\caption{Two cases when $D_a$ and $D_b$ are not parallel}
\label{cas3}
\end{figure}
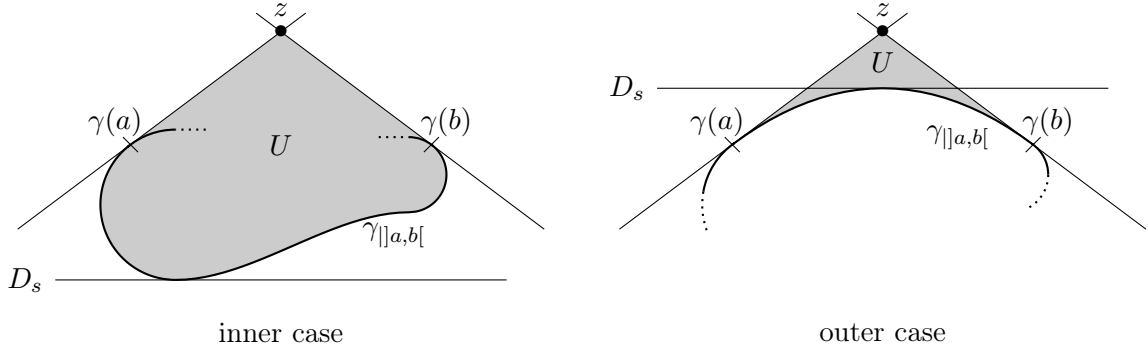

Every case leads to a contradiction, we deduce $I=\cer$ i.e. every tangent line is a support straight line. Thus $\gamma$ surrounds a convex set.
\end{proof}


\section{Example of a minimizer with a self-contact point}\label{example}

In this section we give the example of 
a set $\OO$ such that every minimal curve for \eqref{problem} has necessarily a self-contact point. The confinement $\OO$ will be a thin tubular neighborhood of a curve describing two drops linked by a circular arc with angular amplitude bigger than $\pi$ (see Figure \ref{gouttes}). 
The arguments below use the fact that $\OO$ is not simply connected, and any minimal curve satisfies
\[ \big\{ z \in \RR \ | \ \indc_\gamma(z) = 1 \big\} \not\subset \OO . \]
In this example any minimal curve surrounds a set which is not included in the confinement. Does there exist a simply connected open confinement $\OO$ ensuring that a minimal curve has a self-contact? This is still an open problem.

\begin{figure}[!ht]
\begin{center}

\begin{tikzpicture}
\draw[color = gray!75] ({-3/2},{-sqrt(6)/2}) circle (1)
	({-3/2-2*cos(30+45)},{-sqrt(6)/2+2*sin(30+45)}) circle (1)
	({-3/2-2*cos(-30+45)},{-sqrt(6)/2+2*sin(-30+45)}) circle (1) ;
\draw[thick]
	plot[domain=0:60,samples=100]
		({-3/2-2*cos(30+45) + cos(225+\x)},{-sqrt(6)/2+2*sin(30+45) + sin(225+\x)})
	plot[domain=0:300,samples=100]
		({-3/2+cos(105-\x)},{-sqrt(6)/2+sin(105-\x)})
	plot[domain=0:60,samples=100]
		({-3/2-2*cos(-30+45)+cos(-15+\x)},{-sqrt(6)/2+2*sin(-30+45)+sin(-15+\x)}) --
	({-1-3/2-2*cos(-30+45)+cos(-15+60)},1) ;
\end{tikzpicture}
\qquad
\begin{tikzpicture}[scale=0.5]
\draw[thick]
	plot[domain=0:60,samples=100]
		({-3/2-2*cos(30+45) + cos(225+\x)},{-sqrt(6)/2+2*sin(30+45) + sin(225+\x)})
	plot[domain=0:300,samples=100]
		({-3/2+cos(105-\x)},{-sqrt(6)/2+sin(105-\x)})
	plot[domain=0:60,samples=100]
		({-3/2-2*cos(-30+45)+cos(-15+\x)},{-sqrt(6)/2+2*sin(-30+45)+sin(-15+\x)}) ;
\draw[thick]
	plot[domain=0:60,samples=100]
		({3/2+2*cos(30+45) - cos(225+\x)},{-sqrt(6)/2+2*sin(30+45) + sin(225+\x)})
	plot[domain=0:300,samples=100]
		({3/2-cos(105-\x)},{-sqrt(6)/2+sin(105-\x)})
	plot[domain=0:60,samples=100]
		({3/2+2*cos(-30+45)-cos(-15+\x)},{-sqrt(6)/2+2*sin(-30+45)+sin(-15+\x)}) ;
\draw[thick] plot[domain=0:270,samples=100]
	({(sqrt(2)*(3/2+sqrt(6)/2))*cos(225-\x)},
	{(3/2+sqrt(6)/2)+(sqrt(2)*(3/2+sqrt(6)/2))*sin(225-\x)}) ;
\draw (0,{(3/2+sqrt(6)/2)}) node{linked drops} ;
\end{tikzpicture}
\qquad
\begin{tikzpicture}[scale=0.5]
\filldraw[fill=gray!40, draw=black]
	plot[domain=0:60,samples=100]
		({-3/2-2*cos(30+45) + 0.7*cos(225+\x)},{-sqrt(6)/2+2*sin(30+45) + 0.7*sin(225+\x)}) --
	plot[domain=0:300,samples=100]
		({-3/2+1.3*cos(105-\x)},{-sqrt(6)/2+1.3*sin(105-\x)}) --
	plot[domain=0:60,samples=100]
		({-3/2-2*cos(-30+45)+0.7*cos(-15+\x)},{-sqrt(6)/2+2*sin(-30+45)+0.7*sin(-15+\x)}) --
	plot[domain=0:270,samples=100]
		({(0.3+sqrt(2)*(3/2+sqrt(6)/2))*cos(225-\x)},
		{(3/2+sqrt(6)/2)+(0.3+sqrt(2)*(3/2+sqrt(6)/2))*sin(225-\x)}) --
	plot[domain=-60:0,samples=100]
		({3/2+2*cos(-30+45)-0.7*cos(-15-\x)},{-sqrt(6)/2+2*sin(-30+45)+0.7*sin(-15-\x)}) --
	plot[domain=-300:0,samples=100]
		({3/2-1.3*cos(105+\x)},{-sqrt(6)/2+1.3*sin(105+\x)}) --
	plot[domain=-60:0,samples=100] ({3/2+2*cos(30+45)-0.7*cos(225-\x)},{-sqrt(6)/2+2*sin(30+45)+0.7*sin(225-\x)}) --
	plot[domain=-270:0,samples=100]
		({(-0.3+sqrt(2)*(3/2+sqrt(6)/2))*cos(225+\x)},
		{(3/2+sqrt(6)/2)+(-0.3+sqrt(2)*(3/2+sqrt(6)/2))*sin(225+\x)}) -- cycle ;
\filldraw[fill=white, draw=black]
	plot[domain=0:60,samples=100]
		({-3/2-2*cos(30+45) + cos(225+\x)},{-sqrt(6)/2+2*sin(30+45) + sin(225+\x)}) --
	plot[domain=0:300,samples=100]
		({-3/2+cos(105-\x)},{-sqrt(6)/2+sin(105-\x)}) --
	plot[domain=0:60,samples=100]
		({-3/2-2*cos(-30+45)+cos(-15+\x)},{-sqrt(6)/2+2*sin(-30+45)+sin(-15+\x)}) -- cycle ;
\filldraw[fill=white, draw=black]
	plot[domain=0:60,samples=100]
		({3/2+2*cos(30+45) - cos(225+\x)},{-sqrt(6)/2+2*sin(30+45) + sin(225+\x)}) --
	plot[domain=0:300,samples=100]
		({3/2-cos(105-\x)},{-sqrt(6)/2+sin(105-\x)}) --
	plot[domain=0:60,samples=100]
		({3/2+2*cos(-30+45)-cos(-15+\x)},{-sqrt(6)/2+2*sin(-30+45)+sin(-15+\x)}) -- cycle ;
\draw (0,{1+(3/2+sqrt(6)/2)}) node{tubular}
	(0,{(3/2+sqrt(6)/2)}) node{neighborhood}
	(0,{-1+(3/2+sqrt(6)/2)}) node{$\OO$} ;
\end{tikzpicture}

\end{center}
\caption{A thin confinement around two linked drops}
\label{gouttes}
\end{figure}
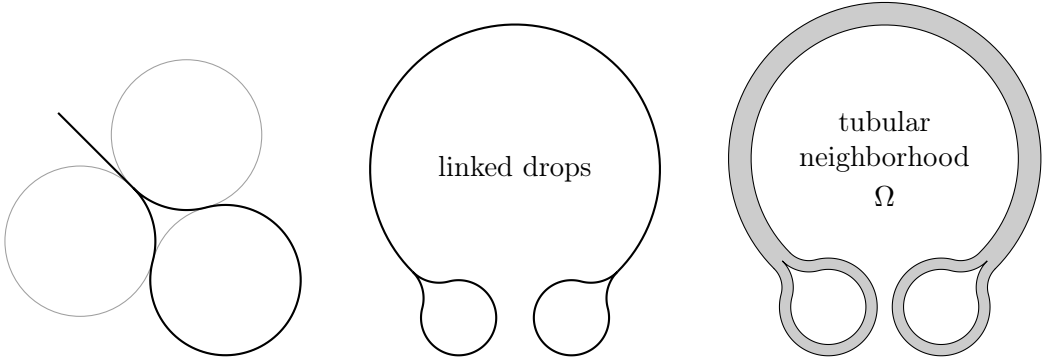

We will first show that any minimal curve surrounds necessarily both drops. Secondly, we assume the existence of a minimal curve without self-contact point, and we construct a deformation that decreases its energy, leading to a contradiction.

We will use an inequality on Jordan curves proved independently in \cite{Bucur2014} and \cite{Ferone2014}:
\begin{thm} \label{inégalité BH}
Let $\gamma$ be a $C^2$ Jordan curve bordering a simply connected subset $\mathcal{U}$. Then, it holds
\[ \WW(\gamma)^2|\mathcal{U}| \geqslant 4\pi^3 .\]
\end{thm}

Consider the curve $\Gamma$ depicted in Figure \ref{gouttes}, and 
let $\OO_\ee$ be a $\ee$-tubular neighborhood of $\Gamma$. 
If we consider a simple minimal curve $\gamma$ surrounding a set $\mathcal{U}$, three cases arise (see Figure \ref{3cases}):

\begin{enumerate}
\item $\mathcal{U} \subset \OO_\ee$. Then $|\mathcal{U}| 
\leqslant |\OO_\ee|$ and so, by Theorem \ref{inégalité BH}, we have that
\[ \WW(\gamma)^2 \geqslant \frac{4\pi^3}{|\OO_\ee|} \stackrel{\ee \to 0}{\longrightarrow} +\infty .\]
Thus, for $\ee$ small enough, $\WW(\gamma) > \WW(\Gamma)$, which contradicts the minimality of $\gamma$.
\item $\mathcal{U}$ surrounds only one drop. Then $\gamma$ has to 
turn about in a
$\ee$-width space. Considering a part of the curve along which an angle of
$\frac\pi 4$ is described in a very small area, a symmetrization argument and a
new call to Theorem \ref{inégalité BH} yields that $\lim_{\ee \to 0} \WW(\gamma) =
+\infty$, and, as in the previous case, $\gamma$ cannot be a minimal curve.
\item $\mathcal{U}$ surrounds both drops, and there exist two branches 
of $\gamma$ in the tubular neighborhood joining both drops (called outer and inner branches).
\end{enumerate}

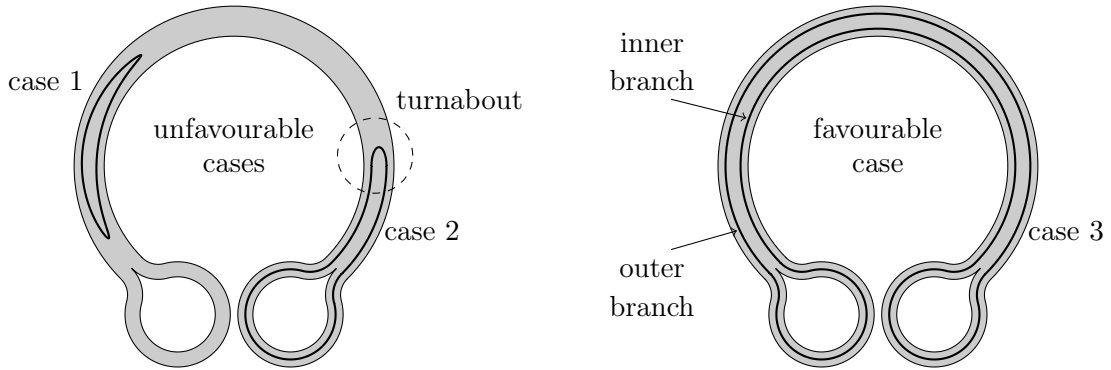
\begin{figure}[!ht]
\begin{center}

\begin{tikzpicture}[scale=0.5]
\filldraw[fill=gray!40, draw=black]
	plot[domain=0:60,samples=100]
		({-3/2-2*cos(30+45) + 0.6*cos(225+\x)},{-sqrt(6)/2+2*sin(30+45) + 0.6*sin(225+\x)}) --
	plot[domain=0:300,samples=100]
		({-3/2+1.4*cos(105-\x)},{-sqrt(6)/2+1.4*sin(105-\x)}) --
	plot[domain=0:60,samples=100]
		({-3/2-2*cos(-30+45)+0.6*cos(-15+\x)},{-sqrt(6)/2+2*sin(-30+45)+0.6*sin(-15+\x)}) --
	plot[domain=0:270,samples=100]
		({(0.4+sqrt(2)*(3/2+sqrt(6)/2))*cos(225-\x)},
		{(3/2+sqrt(6)/2)+(0.4+sqrt(2)*(3/2+sqrt(6)/2))*sin(225-\x)}) --
	plot[domain=-60:0,samples=100]
		({3/2+2*cos(-30+45)-0.6*cos(-15-\x)},{-sqrt(6)/2+2*sin(-30+45)+0.6*sin(-15-\x)}) --
	plot[domain=-300:0,samples=100]
		({3/2-1.4*cos(105+\x)},{-sqrt(6)/2+1.4*sin(105+\x)}) --
	plot[domain=-60:0,samples=100] ({3/2+2*cos(30+45)-0.6*cos(225-\x)},{-sqrt(6)/2+2*sin(30+45)+0.6*sin(225-\x)}) --
	plot[domain=-270:0,samples=100]
		({(-0.4+sqrt(2)*(3/2+sqrt(6)/2))*cos(225+\x)},
		{(3/2+sqrt(6)/2)+(-0.4+sqrt(2)*(3/2+sqrt(6)/2))*sin(225+\x)}) -- cycle ;
\filldraw[fill=white, draw=black]
	plot[domain=0:60,samples=100]
		({-3/2-2*cos(30+45) + cos(225+\x)},{-sqrt(6)/2+2*sin(30+45) + sin(225+\x)}) --
	plot[domain=0:300,samples=100]
		({-3/2+cos(105-\x)},{-sqrt(6)/2+sin(105-\x)}) --
	plot[domain=0:60,samples=100]
		({-3/2-2*cos(-30+45)+cos(-15+\x)},{-sqrt(6)/2+2*sin(-30+45)+sin(-15+\x)}) -- cycle ;
\filldraw[fill=white, draw=black]
	plot[domain=0:60,samples=100]
		({3/2+2*cos(30+45) - cos(225+\x)},{-sqrt(6)/2+2*sin(30+45) + sin(225+\x)}) --
	plot[domain=0:300,samples=100]
		({3/2-cos(105-\x)},{-sqrt(6)/2+sin(105-\x)}) --
	plot[domain=0:60,samples=100]
		({3/2+2*cos(-30+45)-cos(-15+\x)},{-sqrt(6)/2+2*sin(-30+45)+sin(-15+\x)}) -- cycle ;

\draw[thick] plot[domain=0:360, samples=100]
	({(sqrt(2)*(3/2+sqrt(6)/2))*cos(170+40*sin(\x)) + 0.2*cos(\x)},
	{(3/2+sqrt(6)/2)+(sqrt(2)*(3/2+sqrt(6)/2))*sin(170+40*sin(\x)) }) ;
\draw[thick]
	plot[domain=0:45,samples=100]
		({(0.2+sqrt(2)*(3/2+sqrt(6)/2))*cos(0-\x)},
		{(3/2+sqrt(6)/2)+(0.2+sqrt(2)*(3/2+sqrt(6)/2))*sin(0-\x)}) --
	plot[domain=-60:0,samples=100]
		({3/2+2*cos(-30+45)-0.8*cos(-15-\x)},{-sqrt(6)/2+2*sin(-30+45)+0.8*sin(-15-\x)}) --
	plot[domain=-300:0,samples=100]
		({3/2-1.2*cos(105+\x)},{-sqrt(6)/2+1.2*sin(105+\x)}) --
	plot[domain=-60:0,samples=100]
		({3/2+2*cos(30+45)-0.8*cos(225-\x)},{-sqrt(6)/2+2*sin(30+45)+0.8*sin(225-\x)}) --
	plot[domain=-45:0,samples=100]
		({(-0.2+sqrt(2)*(3/2+sqrt(6)/2))*cos(0+\x)},
		{(3/2+sqrt(6)/2)+(-0.2+sqrt(2)*(3/2+sqrt(6)/2))*sin(0+\x)}) -- 
	plot[domain=0:180,samples=100]
		({sqrt(2)*(3/2+sqrt(6)/2)+0.2*cos(180-\x)},{3/2+sqrt(6)/2+0.5*sin(180-\x)}) -- cycle ;
\draw (0,{1+(3/2+sqrt(6)/2)}) node{unfavourable}
	(0,{(3/2+sqrt(6)/2)}) node{cases} ;
\draw (-5,5) node{case $1$} (5,1) node{case $2$} ;
\draw[dashed] (3.75,3) circle (1) ;
\draw (6,4.5) node{turnabout} ;
\end{tikzpicture}
\qquad
\begin{tikzpicture}[scale=0.5]
\filldraw[fill=gray!40, draw=black]
	plot[domain=0:60,samples=100]
		({-3/2-2*cos(30+45) + 0.6*cos(225+\x)},{-sqrt(6)/2+2*sin(30+45) + 0.6*sin(225+\x)}) --
	plot[domain=0:300,samples=100]
		({-3/2+1.4*cos(105-\x)},{-sqrt(6)/2+1.4*sin(105-\x)}) --
	plot[domain=0:60,samples=100]
		({-3/2-2*cos(-30+45)+0.6*cos(-15+\x)},{-sqrt(6)/2+2*sin(-30+45)+0.6*sin(-15+\x)}) --
	plot[domain=0:270,samples=100]
		({(0.4+sqrt(2)*(3/2+sqrt(6)/2))*cos(225-\x)},
		{(3/2+sqrt(6)/2)+(0.4+sqrt(2)*(3/2+sqrt(6)/2))*sin(225-\x)}) --
	plot[domain=-60:0,samples=100]
		({3/2+2*cos(-30+45)-0.6*cos(-15-\x)},{-sqrt(6)/2+2*sin(-30+45)+0.6*sin(-15-\x)}) --
	plot[domain=-300:0,samples=100]
		({3/2-1.4*cos(105+\x)},{-sqrt(6)/2+1.4*sin(105+\x)}) --
	plot[domain=-60:0,samples=100]
		({3/2+2*cos(30+45)-0.6*cos(225-\x)},{-sqrt(6)/2+2*sin(30+45)+0.6*sin(225-\x)}) --
	plot[domain=-270:0,samples=100]
		({(-0.4+sqrt(2)*(3/2+sqrt(6)/2))*cos(225+\x)},
		{(3/2+sqrt(6)/2)+(-0.4+sqrt(2)*(3/2+sqrt(6)/2))*sin(225+\x)}) -- cycle ;
\filldraw[fill=white, draw=black]
	plot[domain=0:60,samples=100]
		({-3/2-2*cos(30+45) + cos(225+\x)},{-sqrt(6)/2+2*sin(30+45) + sin(225+\x)}) --
	plot[domain=0:300,samples=100]
		({-3/2+cos(105-\x)},{-sqrt(6)/2+sin(105-\x)}) --
	plot[domain=0:60,samples=100]
		({-3/2-2*cos(-30+45)+cos(-15+\x)},{-sqrt(6)/2+2*sin(-30+45)+sin(-15+\x)}) -- cycle ;
\filldraw[fill=white, draw=black]
	plot[domain=0:60,samples=100]
		({3/2+2*cos(30+45) - cos(225+\x)},{-sqrt(6)/2+2*sin(30+45) + sin(225+\x)}) --
	plot[domain=0:300,samples=100]
		({3/2-cos(105-\x)},{-sqrt(6)/2+sin(105-\x)}) --
	plot[domain=0:60,samples=100]
		({3/2+2*cos(-30+45)-cos(-15+\x)},{-sqrt(6)/2+2*sin(-30+45)+sin(-15+\x)}) -- cycle ;
\draw[thick]
	plot[domain=0:60,samples=100]
		({-3/2-2*cos(30+45) + 0.8*cos(225+\x)},{-sqrt(6)/2+2*sin(30+45) + 0.8*sin(225+\x)}) --
	plot[domain=0:300,samples=100]
		({-3/2+1.2*cos(105-\x)},{-sqrt(6)/2+1.2*sin(105-\x)}) --
	plot[domain=0:60,samples=100]
		({-3/2-2*cos(-30+45)+0.8*cos(-15+\x)},{-sqrt(6)/2+2*sin(-30+45)+0.8*sin(-15+\x)}) --
	plot[domain=0:270,samples=100]
		({(0.2+sqrt(2)*(3/2+sqrt(6)/2))*cos(225-\x)},
		{(3/2+sqrt(6)/2)+(0.2+sqrt(2)*(3/2+sqrt(6)/2))*sin(225-\x)}) --
	plot[domain=-60:0,samples=100]
		({3/2+2*cos(-30+45)-0.8*cos(-15-\x)},{-sqrt(6)/2+2*sin(-30+45)+0.8*sin(-15-\x)}) --
	plot[domain=-300:0,samples=100]
		({3/2-1.2*cos(105+\x)},{-sqrt(6)/2+1.2*sin(105+\x)}) --
	plot[domain=-60:0,samples=100]
	({3/2+2*cos(30+45)-0.8*cos(225-\x)},{-sqrt(6)/2+2*sin(30+45)+0.8*sin(225-\x)}) --
	plot[domain=-270:0,samples=100]
		({(-0.2+sqrt(2)*(3/2+sqrt(6)/2))*cos(225+\x)},
		{(3/2+sqrt(6)/2)+(-0.2+sqrt(2)*(3/2+sqrt(6)/2))*sin(225+\x)}) -- cycle ;
\draw (0,{1+(3/2+sqrt(6)/2)}) node{favourable}
	(0,{(3/2+sqrt(6)/2)}) node{case} ;
\draw (5,1) node{case $3$} ;
\draw (-6,6) node{inner} (-6,5) node{branch} (-6,0) node{outer} (-6,-1) node{branch} ;
\draw[->] (-5.5,4.5) -- (-3.5,4) ;
\draw[->] (-5.5,0.5) -- (-3.75,1) ;
\end{tikzpicture}

\end{center}
\caption{Three cases for a minimal curve}
\label{3cases}
\end{figure}

Focusing on the third case, as the tubular neighborhood has an angular amplitude bigger than $\pi$, we can find two non parallel straight lines tangent to the inner branch at points $x_\ell$ and $x_r$ near each drop (see Figure \ref{slide}). Consider the intersection point of these lines and a small dilation centered at this point. The part of the inner branch between both lines remains tangent to these lines under the dilation, and its energy decreases. We can reconnect this dilated arc with points $x_\ell$ and $x_r$ using tangential segments. We choose these tangential segments in order to have an admissible curve with a smaller energy than the minimum. This leads to a contradiction, thus any minimal curve has at least a self-contact, which yields the following result:

\begin{thm}\label{thkey}
For $\ee$ small enough, 
every minimal curve in $\OO_\ee$ has a self-contact point.
\end{thm}

\begin{figure}[!ht]
\begin{center}

\begin{tikzpicture}[scale=0.7]
\fill[color=gray!40]
	plot[domain=0:270, samples=100] ({3*cos(225-\x)},{3*sin(225-\x)}) --
	plot[domain=0:270, samples=100] ({2*cos(-45+\x)},{2*sin(-45+\x)}) -- cycle ;
\draw
	plot[domain=0:270, samples=100] ({3*cos(225-\x)},{3*sin(225-\x)})
	plot[domain=0:270, samples=100] ({2*cos(-45+\x)},{2*sin(-45+\x)}) ;
\draw[thick]
	plot[domain=0:270, samples=100] ({2.7*cos(225-\x)},{2.7*sin(225-\x)})
	plot[domain=0:270, samples=100] ({2.2*cos(-45+\x)},{2.2*sin(-45+\x)}) ;
\draw ({2.2*cos(-15)+3*cos(75)},{2.2*sin(-15)+3*sin(75)}) --
	({2.2*cos(-15)-3*cos(75)},{2.2*sin(-15)-3*sin(75)})
	({2.2*cos(195)+3*cos(105)},{2.2*sin(195)+3*sin(105)}) --
	({2.2*cos(195)-3*cos(105)},{2.2*sin(195)-3*sin(105)}) ;
\foreach \t in {10,20,...,171}
	{\draw[->] ({2.3*cos(\t)},{2.3*sin(\t)}) --
		+({0.3*cos(90+(15/105)*(\t-90)},{0.3*sin(90+(15/105)*(\t-90)}) ; }
\draw (-2.15,-0.5) node{$-$} +(0.5,0) node{$x_\ell$} ;
\draw (2.15,-0.5) node{$-$} +(-0.5,0) node{$x_r$} ;
\end{tikzpicture}
\qquad
\begin{tikzpicture}[scale=0.7]
\fill[color=gray!40]
	plot[domain=0:270, samples=100] ({3*cos(225-\x)},{3*sin(225-\x)}) --
	plot[domain=0:270, samples=100] ({2*cos(-45+\x)},{2*sin(-45+\x)}) -- cycle ;
\draw
	plot[domain=0:270, samples=100] ({3*cos(225-\x)},{3*sin(225-\x)})
	plot[domain=0:270, samples=100] ({2*cos(-45+\x)},{2*sin(-45+\x)}) ;
\draw[thick]
	plot[domain=0:270, samples=100] ({2.7*cos(225-\x)},{2.7*sin(225-\x)})
	plot[domain=0:30, samples=100] ({2.2*cos(-45+\x)},{2.2*sin(-45+\x)}) --
	plot[domain=30:240, samples=100] ({2.3*cos(-45+\x)},{0.4+2.3*sin(-45+\x)}) --
	plot[domain=240:270, samples=100] ({2.2*cos(-45+\x)},{2.2*sin(-45+\x)}) ;
\draw ({2.2*cos(-15)+3*cos(75)},{2.2*sin(-15)+3*sin(75)}) --
	({2.2*cos(-15)-3*cos(75)},{2.2*sin(-15)-3*sin(75)})
	({2.2*cos(195)+3*cos(105)},{2.2*sin(195)+3*sin(105)}) --
	({2.2*cos(195)-3*cos(105)},{2.2*sin(195)-3*sin(105)}) ;
\end{tikzpicture}

\end{center}
\caption{Dilation of a part of a minimal curve}
\label{slide}
\end{figure}
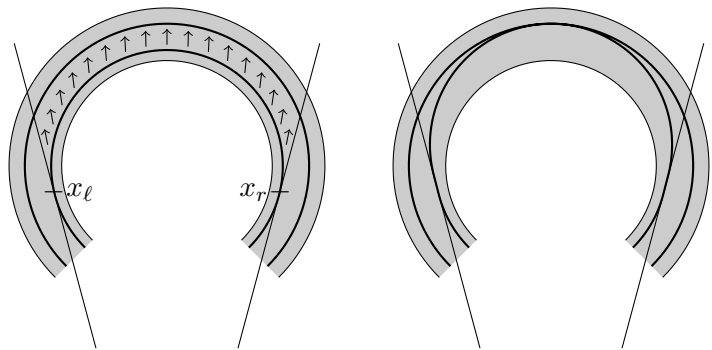




\begin{thebibliography}{10}

\bibitem{Ambrosio2000}
L.~Ambrosio, N.~Fusco, and D.~Pallara.
\newblock {\em Functions of Bounded Variation and Free Discontinuity Problems.}
\newblock Oxford University Press, 2000.

\bibitem{Ambrosio2003}
L.~Ambrosio and S.~Masnou.
\newblock A direct variational approach to a problem arising in image
  reconstruction.
\newblock {\em Interfaces and Free Boundaries}, 5:63--81, 2003.

\bibitem{Arreaga2002}
G.~Arreaga, R.~Capovilla, C.~Chryssomalakos, and J.~Guven.
\newblock Area-constrained planar elastica.
\newblock {\em Physical Review E}, 65(3), 2002.

\bibitem{Aubin1998}
T.~Aubin.
\newblock {\em Some Nonlinear Problems in Riemannian Geometry}.
\newblock Springer Monographs in Mathematics. Springer, 1998.

\bibitem{Avvakumov2013}
S.~Avvakumov, O.~Karpenkov, and A.~Sassinsky.
\newblock Euler elasticae in the plane and the {W}hitney-{G}raustein theorem.
\newblock {\em Russian Journal of Mathematical Physics}, 20(3):257--267, July
  2013.

\bibitem{Bellettini1993}
G.~Bellettini, G.~{Dal Maso}, and M.~Paolini.
\newblock Semi-continuity and relaxation properties of a curvature depending
  functional in {2D}.
\newblock {\em Annali della Scuola Normale Superiore di Pisa, Classe di
  Scienze}, 20(2):247--297, 1993.

\bibitem{Bellettini2004}
G.~Bellettini and L.~Mugnai.
\newblock Characterization and representation of the lower semicontinuous
  envelope of the elastica functional.
\newblock {\em Annales de l'Institut Henri Poincar{\'e}}, 21:839--880, 2004.

\bibitem{Bellettini2007}
G.~Bellettini and L.~Mugnai.
\newblock A varifolds representation of the relaxed elastica functional.
\newblock {\em Journal of Convex Analysis}, 14(3):543--564, 2007.

\bibitem{Bredies2015}
K.~Bredies, T.~Pock, and B.~Wirth.
\newblock A convex, lower semi-continuous approximation of {E}uler's elastica
  energy.
\newblock {\em SIAM Journal on Mathematical Analysis}, 47(1):566--613,  2015.

\bibitem{Bretin2011}
E.~Bretin, J.-O. Lachaud, and E.~Oudet.
\newblock Regularization of discrete contour by {W}illmore energy.
\newblock {\em Journal of Mathematical Imaging and Vision}, 40(2):214--229,
  2011.

\bibitem{Bucur2014}
D.~Bucur and A.~Henrot.
\newblock A new isoperimetric inequality for the elasticae.
\newblock ArXiv e-print 1412.4536, 2014.

\bibitem{Cao2011}
F.~Cao, Y.~Gousseau, S.~Masnou, and P.~P{\'e}rez.
\newblock Geometrically guided exemplar-based inpainting.
\newblock {\em SIAM Journal on Imaging Sciences}, 4(4):1143--1179, 2011.

\bibitem{Chan2002}
T.~Chan, S.~H. Kang, and J.~Shen.
\newblock {E}uler's elastica and curvature based inpaintings.
\newblock {\em SIAM Journal on Applied Mathematics}, 63:564--592, 2002.

\bibitem{Citti2006}
G.~Citti and A.~Sarti.
\newblock A cortical based model of perceptual completion in the
  roto-translation space.
\newblock {\em Journal of Mathematical Imaging and Vision}, 24:307--326, 2006.

\bibitem{Coope1992}
I.~Coope.
\newblock Curve interpolation with nonlinear spiral splines.
\newblock {\em IMA Journal of numerical analysis}, 13:327--341, 1992.

\bibitem{Dondl2011}
P.~Dondl, L.~Mugnai, and M.~R{\"o}ger.
\newblock Confined elastic curves.
\newblock {\em SIAM Journal on Applied Mathematics}, 71(6):2205--2226, 2011.

\bibitem{2015arXiv150701856D}
P.~W. {Dondl}, A.~{Lemenant}, and S.~{Wojtowytsch}.
\newblock {Phase field models for thin elastic structures with topological
  constraint}.
\newblock ArXiv e-print 1507.01856, 2015.

\bibitem{EsedogluShen}
S.~Esedoglu and J.~Shen.
\newblock Digital image inpainting by the {M}umford-{S}hah-{E}uler image model.
\newblock {\em European J. Appl. Math.}, 13:353--370, 2002.

\bibitem{Evans1992}
L.~C. Evans and R.~Gariepy.
\newblock {\em Measure Theory and Fine Properties of Functions.}
\newblock CRC Press, 1992.

\bibitem{Ferone2014}
V.~Ferone, B.~Kawohl, and C.~Nitsch.
\newblock The elastica problem under area constraint.
\newblock Preprint, 2014.

\bibitem{Forsythe1973}
G.~Forsythe and E.~Lee.
\newblock Variational study of nonlinear spline curves.
\newblock {\em SIAM Review}, 15:120--133, 1973.

\bibitem{Giaquinta2012}
M.~Giaquinta and G.~Modica.
\newblock {\em Mathematical Analysis, Fundations and Advanced Techniques for
  Functions of Several Variables}.
\newblock Birkhauser, 2012.

\bibitem{Hebey2000}
E.~Hebey.
\newblock {\em Nonlinear Analysis on Manifolds: Sobolev Spaces and
  Inequalities}.
\newblock Lecture notes. American Mathematical Society, 2000.

\bibitem{Horn1983}
B.~Horn.
\newblock The curve of least energy.
\newblock {\em Association for Computing Machinery - Transactions on
  Mathematical Software}, 9(4):441--460, 1983.

\bibitem{Koiso1992}
N.~Koiso.
\newblock Elastica in a {R}iemannian submanifold.
\newblock {\em Osaka Journal of Mathematics}, 29:539--543, 1992.

\bibitem{Langer1984a}
J.~Langer and D.~Singer.
\newblock Knotted elastic curves in ${R}^3$.
\newblock {\em Kangwon-Kyunggi Mathematical Journal}, 5(2):113--119, 1984.

\bibitem{Langer1984}
J.~Langer and D.~Singer.
\newblock The total squared curvature of closed curves.
\newblock {\em Journal of Differential Geometry}, 20:1--22, 1984.

\bibitem{Langer1985}
J.~Langer and D.~Singer.
\newblock Curve straightening and a minimax argument for closed elastic curves.
\newblock {\em Topology}, 24(1):75--88, 1985.

\bibitem{Linner89}
A.~Linn{\'e}r.
\newblock Some properties of the curve straightening flow in the plane.
\newblock {\em Trans. Amer. Math. Soc.}, 314(2):605--618, 1989.

\bibitem{Masnou2006}
S.~Masnou and J.~M. Morel.
\newblock On a variational theory of image amodal completion.
\newblock {\em Rendiconti del Seminario Matematico della Universit{\`a} di
  Padova}, 116:211--252, 2006.

\bibitem{Miura2015}
T.~Miura.
\newblock Singular perturbation by bending for an adhesive obstacle problem.
\newblock Preprint, 2015.

\bibitem{Mumford}
D.~Mumford.
\newblock Elastica and computer vision.
\newblock In C.~Bajaj, editor, {\em Algebraic Geometry and its Applications},
  pages 491--506. Springer-Verlag, New-York, 1994.

\bibitem{MumfordNitzberg}
M.~Nitzberg and D.~Mumford.
\newblock The 2.1-{D} {S}ketch.
\newblock In {\em Proc. 3rd Int. Conf. on Computer Vision}, pages 138--144,
  Osaka, Japan, 1990.

\bibitem{Okabe2007}
S.~Okabe.
\newblock The motion of elastic planar closed curves under the area preserving
  condition.
\newblock {\em Indiana University Mathematics Journal}, 56(4):1871--1912, 2007.

\bibitem{Olischlaeger2009}
N.~Olischl{\"a}ger and M.~Rumpf.
\newblock Two step time discretization of {W}illmore flow. 
\newblock {\em Lecture Notes in Computer Science}, 5654:278--292, 2009.

\bibitem{Sachkov2008}
Y.~Sachkov.
\newblock Maxwell strata in the {E}uler elastic problem.
\newblock {\em Journal of Dynamical and Control Systems}, 14(2):169--234, 
  2008.

\bibitem{Sachkov2012}
Y.~Sachkov.
\newblock Closed {E}uler elasticae.
\newblock {\em Proceedings of the Steklov Institute of Mathematics},
  278:218--232, 2012.

\bibitem{SchoenemannKMC12}
T.~Schoenemann, F.~Kahl, S.~Masnou, and D.~Cremers.
\newblock A linear framework for region-based image segmentation and inpainting
  involving curvature penalization.
\newblock {\em International Journal of Computer Vision}, 99(1):53--68, 2012.

\bibitem{Schoenemann-et-al-tip11}
T.~Schoenemann, S.~Masnou, and D.~Cremers.
\newblock The elastic ratio: Introducing curvature into ratio-based globally
  optimal image segmentation.
\newblock {\em IEEE Trans. Image Processing}, 20(9):2565--2581, 2011.

\bibitem{ulen15}
J.~Ul{\'e}n, P.~Strandmark, and F.~Kahl.
\newblock Shortest paths with higher-order regularization.
\newblock {\em IEEE Transactions on Pattern Analysis and Machine Intelligence},
  2015.

\end{thebibliography}
\end{document}